\documentclass{amsart}
\usepackage{amsmath}
\usepackage{amsthm}
\usepackage{amssymb}
\usepackage{appendix}
\usepackage{bm}
\usepackage{dsfont}
\usepackage{enumitem}\setlist[enumerate]{nosep,label=\textnormal{(\arabic*)}}
\usepackage{eucal}
\usepackage[T1]{fontenc}
\usepackage{graphicx}
\usepackage[hidelinks]{hyperref}
\usepackage[capitalise]{cleveref}
\usepackage{mathrsfs}
\usepackage{mathtools}
\usepackage{nameref}
\usepackage{physics}
\usepackage{stmaryrd}
\usepackage{thmtools}
\usepackage{xcolor}
\usepackage{xparse}

\makeatletter
\let\@wraptoccontribs\wraptoccontribs
\makeatother

\makeatletter
\DeclareFontFamily{OMX}{MnSymbolE}{}
\DeclareSymbolFont{MnLargeSymbols}{OMX}{MnSymbolE}{m}{n}
\SetSymbolFont{MnLargeSymbols}{bold}{OMX}{MnSymbolE}{b}{n}
\DeclareFontShape{OMX}{MnSymbolE}{m}{n}{
    <-6>  MnSymbolE5
   <6-7>  MnSymbolE6
   <7-8>  MnSymbolE7
   <8-9>  MnSymbolE8
   <9-10> MnSymbolE9
  <10-12> MnSymbolE10
  <12->   MnSymbolE12
}{}
\DeclareFontShape{OMX}{MnSymbolE}{b}{n}{
    <-6>  MnSymbolE-Bold5
   <6-7>  MnSymbolE-Bold6
   <7-8>  MnSymbolE-Bold7
   <8-9>  MnSymbolE-Bold8
   <9-10> MnSymbolE-Bold9
  <10-12> MnSymbolE-Bold10
  <12->   MnSymbolE-Bold12
}{}

\let\llangle\@undefined
\let\rrangle\@undefined
\DeclareMathDelimiter{\llangle}{\mathopen}%
                     {MnLargeSymbols}{'164}{MnLargeSymbols}{'164}
\DeclareMathDelimiter{\rrangle}{\mathclose}%
                     {MnLargeSymbols}{'171}{MnLargeSymbols}{'171}
\makeatother

\usepackage{tikz}
\usetikzlibrary{cd}
\usetikzlibrary{positioning, angles, quotes}
\newcommand{\nospacepunct}[1]{\makebox[0pt][l]{\,#1}} 

\theoremstyle{plain}
\newtheorem{thm}{Theorem}[section]\RenewCommandCopy{\theHthm}{\thethm}
\newtheorem{cor}[thm]{Corollary}\RenewCommandCopy{\theHcor}{\thecor}
\newtheorem{prop}[thm]{Proposition}\RenewCommandCopy{\theHprop}{\theprop}
\newtheorem{lem}[thm]{Lemma}\RenewCommandCopy{\theHlem}{\thelem}
\newtheorem{claim}[thm]{Claim}\RenewCommandCopy{\theHclaim}{\theclaim}
\RenewCommandCopy{\theHexer}{\theexer}
\RenewCommandCopy{\theHq}{\theq}
\newtheorem{conj}[thm]{Conjecture}\RenewCommandCopy{\theHconj}{\theconj}

\theoremstyle{definition}
\newtheorem{defn}[thm]{Definition}\RenewCommandCopy{\theHdefn}{\thedefn}
\RenewCommandCopy{\theHex}{\theex}

\theoremstyle{remark}
\newtheorem{rem}[thm]{Remark}\RenewCommandCopy{\theHrem}{\therem}
\RenewCommandCopy{\theHconv}{\theconv}

\Crefname{thm}{Theorem}{Theorems}
\Crefname{lem}{Lemma}{Lemmas}
\Crefname{defn}{Definition}{Definitions}
\Crefname{claim}{Claim}{Claims}
\Crefname{conj}{Conjecture}{Conjectures}
\Crefname{ex}{Example}{Examples}
\Crefname{prop}{Proposition}{Propositions}

\newcommand{\GL}{\mathrm{GL}}

\newcommand{\Sp}{\mathrm{Sp}}

\newcommand{\rk}{\mathrm{rk}}

\newcommand{\C}{\mathbb{C}}
\newcommand{\F}{\mathbb{F}}
\newcommand{\K}{\mathbb{K}}
\newcommand{\N}{\mathbb{N}}
\newcommand{\Q}{\mathbb{Q}}
\newcommand{\R}{\mathbb{R}}
\newcommand{\Z}{\mathbb{Z}}

\newcommand{\inv}{^{-1}}

\DeclareMathOperator{\Aut}{Aut}
\DeclareMathOperator{\cd}{cd}

\DeclareMathOperator{\gr}{\mathfrak{gr}}

\DeclareMathOperator{\im}{im}

\DeclareMathOperator{\Inn}{Inn}

\DeclareMathOperator{\Mat}{Mat}
\DeclareMathOperator{\Mod}{Mod}

\DeclareMathOperator{\Out}{Out}

\newcounter{andrewcomments}

\newcounter{samcomments}

\title{Outer automorphism groups and the Atiyah Conjecture}

\author{Sam P.\ Fisher}
\address{Instituto de Ciencias Matem\'aticas, CSIC-UAM-UC3M-UCM, Madrid, Spain}
\email{samuel.fisher@icmat.es}

\author{Andrew Ng}
\address{Mathematisches Institut, Universität Bonn, Endenicher Allee 60, 53115 Bonn, Germany}
\email{\href{mailto:clan@math.uni-bonn.de}{clan@math.uni-bonn.de}}

\begin{document}

\maketitle

\vspace{-2em}

\begin{center}
    \emph{With an appendix by Sam P.\ Fisher}
\end{center}

\begin{abstract}
    Let $G$ be the fundamental group of a compact surface, a finitely generated free group, or more generally a finitely generated right-angled Artin group. We prove that the von Neumann dimension function of $\Out(G)$ is valued in a discrete subgroup of $\Q$. This is accomplished by establishing the Strong Atiyah Conjecture for a torsion-free subgroup of $\Out(G)$ of finite index. We also prove that for every field $\K$, there exists a torsion-free subgroup $H \leqslant \Out(G)$ of finite index such that $\K[H]$ embeds into a division ring, and hence satisfies the Zero Divisor Conjecture. These results are obtained by establishing analogous ones for a suitable open subgroup of $\Out(\mathbf G)$ and its completed group algebra, where $\mathbf G$ denotes the pro-$p$ completion of $G$. In an appendix, the first author shows that an automorphism of a free nilpotent group is inner if and only if it induces an inner automorphism of its pro-$p$ completion.
\end{abstract}

\section{Introduction}\label{sec:intro}

The general linear groups $\GL_m(\Z)$, outer automorphism groups of free groups $\Out(F_n)$, and mapping class group of closed surfaces $\Mod(\Sigma_g)$ play a central role in geometric group theory. A fruitful guiding principle in their study is to compare and contrast the three families by attempting to understand what properties hold for all three of them, and when and why some property holds for some family but not the other. The groups $\GL_m(\Z)$ are the most classical and well understood; determining which of its properties are shared by $\Out(F_n)$ and $\Mod(\Sigma_g)$ has led to many widely studied open problems in geometric group theory. For instance, Kazhdan proved that $\GL_m(\Z)$ has property (T) for all $m \geqslant 3$ when he introduced the property \cite{Kazhdan1967}, and it was only recently proven that $\Out(F_n)$ has property (T) for $n \geqslant 4$ \cite{KalubaKielakNowak_T,Nitsche_AutF4propT_2023}. Whether $\Mod(\Sigma_g)$ has property (T) for sufficiently large genus $g$ is a central open question in the study of the mapping class group. In another direction, $\GL_m(\Z)$ is by definition linear, while $\Out(F_n)$ is known to be non-linear for $n \geqslant 4$ \cite{FormanekProcesi1992}; the linearity of $\Mod(\Sigma_g)$ is again an important open problem (it is linear for $g = 2$ \cite{BigelowBudney2001}). As a final example, it is known that the conjugacy problem is solvable for both $\GL_m(\Z)$ \cite{Grunewald_conjugacy} and $\Mod(\Sigma_g)$ \cite{Hemion1979}, but is an open problem for $\Out(F_n)$ for $n \geqslant 4$ (it is solvable for $n \leqslant 3$ \cite{DahmaniFrancavigliaMartinoTouikan_ConjOutF3_2025}).

In this article we study the Atiyah Conjecture on $L^2$-invariants and the Kaplansky Zero Divisor Conjecture for group algebras of $\GL_m(\Z)$, $\Out(F_n)$, and $\Mod(\Sigma_g)$. We will also work with outer automorphism groups of finitely generated right-angled Artin groups (or RAAGs), which generalise $\Out(F_n)$ and have gained significant attention. As with the examples above, these conjectures are mostly settled in the case of $\GL_m(\Z)$, so our focus will be on the outer automorphism groups of surfaces, free groups, and RAAGs.

\subsection{The Atiyah and Zero Divisor Conjectures}\label{subsec:atiyah_ZD}
We denote by $\Mat(\C[G])$ the set of finite matrices over the complex group algebra $\C[G]$. For every countable group $G$, there is a rank function 
\[
    \rk_{\mathcal N(G)} \colon \Mat(\C[G]) \to \R_{\geqslant 0}
\]
called the \emph{von Neumann rank function} of $G$. We will define it carefully in \cref{subsec:atiyah}. It satisfies certain basic properties any rank function ought to: the von Neumann rank of the $n$-by-$n$ identity matrix is $n$ and is additive under the diagonal sum of matrices. Importantly, the $L^2$-Betti numbers of a group $G$ acting cocompactly on a CW complex are computed using $\rk_{\mathcal N(G)}$. 

In his seminal work on $L^2$-invariants \cite{Atiyah_OGL2}, Atiyah introduced the $L^2$-Betti numbers of manifolds with a cocompact group action and computed them for certain examples. Each of his computations yielded rational numbers, even though there is nothing in the definition of $L^2$-Betti numbers that implies this should be the case. This prompted Atiyah to ask whether there are examples of cocompact $G$-manifolds with irrational $L^2$-Betti numbers. While Austin \cite{Austin2013} and Grabowski \cite{Grabowski2014} were the first to give examples of groups $G$ such that the von Neumann rank function $\rk_{\mathcal N(G)}$ takes irrational values, Pichot, Schick, and Zuk were the first to definitively resolve Atiyah's problem by giving an example of a compact $7$-manifold such that $b_3^{(2)}(\widetilde{M})$ is transcendental \cite{PichotSchickZuk}.

The examples of groups yielding irrational $L^2$-Betti numbers all have torsion subgroups of unbounded order. For groups $G$ with a bound on the orders of their torsion elements, the question of what possible values $\rk_{\mathcal N(G)}$ can take has remained an active area of research, and has been reformulated into the following open problem known as the Strong Atiyah Conjecture.

\begin{conj}[The Strong Atiyah Conjecture over $\K$]\label{conj:SAC}
    Let $\K$ be a subfield of $\C$, let $G$ be a countable group with a bound on the orders of its finite subgroups, and let $\mathrm{lcm}(G)$ denote the least common multiple of the orders of the finite subgroups of $G$. For all finite matrices $A$ over $\K[G]$, we have $\rk_{\mathcal N(G)}(A) \in \frac{1}{\mathrm{lcm}(G)} \Z$.
\end{conj}

The Strong Atiyah Conjecture over $\C$ is known for many classes of groups, including all locally indicable groups \cite{JaikinLopezStrongAtiyah2020}, braid groups \cite{LinnellSchick_AtiyahExt}, elementary amenable groups \cite{LinnellDivRings93}, virtually compact special groups \cite{Schreve_AtiyahVCS}, and $3$-manifold groups \cite{FriedlLuck_euler,KielakLinton_3mfldAtiyah}. \cref{conj:SAC} is also known to be stable under taking free products of groups satisfying it \cite{SanchezPeralta2025}.

The Strong Atiyah Conjecture has many nice applications in group theory. One of the best known is to Kaplansky's Zero Divisor Conjecture in characteristic zero: if $G$ is torsion-free and satisfies the Strong Atiyah Conjecture over $\K$, then the group algebra $\K[G]$ is a domain. Another interesting application is that groups of homological dimension one satisfying \cref{conj:SAC} are locally free \cite{KrophollerLinnellLueck2009} (it is a well known conjecture that this should hold without the Atiyah Conjecture assumption). Recently, Jaikin-Zapirain and Linton used the Strong Atiyah Conjecture to prove coherence of one-relator groups, and showed that groups of cohomological dimension two satisfying \cref{conj:SAC} are homologically coherent provided they have vanishing second $L^2$-Betti number \cite{JaikinLinton_coherence}. In the results of \cite{KrophollerLinnellLueck2009,JaikinLinton_coherence}, the only essential consequence of the Strong Atiyah Conjecture used is discreteness of $\rk_{\mathcal N(G)}$, not its specific values. This leads to the following weakened version of the conjecture.

\begin{conj}[The Weak Atiyah Conjecture over $\K$]\label{conj:WAC}
    Let $\K$ be a subfield of $\C$ and let $G$ be a countable group with a bound on the orders of its finite subgroups. There is some $n \in \Z$ such that $\rk_{\mathcal N(G)}(A) \in \frac{1}{n} \Z$ for all finite matrices $A$ over $\K[G]$.
\end{conj}

The name ``Weak Atiyah Conjecture'' is taken from \cite{JaikinLintonSanchez_OneRelProd}, where \cref{conj:WAC} is used to prove results on coherence of groups and group algebras. This form of the Atiyah conjecture is also discussed in L\"uck's book: if $\rk_{\mathcal N(G)}(A)$ lies in an additive subgroup $\Z \leqslant \Lambda \leqslant \Q$ for all matrices $A$ over $\K[G]$, then $G$ satisfies the \emph{Atiyah Conjecture of order $\Lambda$ with coefficients in $\K$} (see \cite[Conjecture 10.3]{Luck02}).

The question of whether \cref{conj:SAC} passes to overgroups of finite index is very delicate and not known to hold in general. On the other hand this holds trivially for \cref{conj:WAC}, since $\rk_{\mathcal N(G)} = \frac{1}{[G:H]}\rk_{\mathcal N(H)}$ whenever $H \leqslant G$ is a subgroup of finite index. This makes the Weak Atiyah Conjecture a much more flexible tool in applications. For example, while it is not known whether $\GL_m(\Z)$ satisfies \cref{conj:SAC}, it has a finite-index subgroup that does by results of Farkas--Linnell \cite{FarkasLinnell2006} and Jaikin-Zapirain \cite{JaikinZapirain_basechange}, and therefore the general linear groups $\GL_m(\Z)$ satisfy \cref{conj:WAC}. In fact, this is the case for any finitely generated group that is linear over a field of characteristic zero, providing a rich source of groups satisfying the Weak Atiyah Conjecture.

While the $L^2$-invariants of $\GL_m(\Z)$ are quite well understood, much less is known for $\Out(F_n)$ and $\Mod(\Sigma_g)$. Most results in this direction concern the vanishing and non-vanishing of their $L^2$-Betti numbers. For instance, by \cite[Theorem D]{AbertBergeronFraczykGaboriau_torsion}, the $L^2$-Betti numbers of $\Mod(\Sigma_g)$ vanish in some range of values depending on $g$. Less is known about $\Out(F_n)$: it has a non-vanishing $L^2$-Betti number in degree equal to $\cd_\Q(\Out(F_n))$ \cite{GaboriauNous_topDimL2_2021} and it has vanishing first $L^2$-Betti number for $n \geqslant 3$. In our first result we prove in particular that mapping class groups and outer automorphism groups of free groups satisfy \cref{conj:WAC}. This has the immediate consequence that all $L^2$-Betti numbers of these groups are rational.

\begin{thm}\label{thm:intro_atiyah}
    Let $A_\Gamma$ be the RAAG on the finite simplicial graph $\Gamma$ and let $\Sigma_{g,b,p}$ be a compact surface of genus $g$ with $b \geqslant 0$ boundary components and $p \geqslant 0$ marked points. There are torsion-free subgroups of finite index in $\Out(A_\Gamma)$ and $\Mod(\Sigma_{g,b,p})$ each satisfying the Strong Atiyah Conjecture over $\C$. In particular, $\Out(A_\Gamma)$ and $\Mod(\Sigma_{g,b,p})$ satisfy the Weak Atiyah Conjecture over $\C$.
\end{thm}

The torsion-free finite-index subgroup appearing in \cref{thm:intro_atiyah} is the kernel of the natural map $\Out(A_\Gamma) \to \GL_n(\F_p)$ for an odd prime $p$, where $n$ is the number of vertices in $\Gamma$ (or similarly the kernel of $\Mod(\Sigma_g) \to \Sp_{2g}(\F_p)$). As such, we can also conclude that $\Out(A_\Gamma)$ satisfies the Atiyah Conjecture of order $\frac{1}{|\GL_n(\F_p)|} \Z$ with coefficients in $\C$ (see \cite[Conjecture 10.3]{Luck02}). In the specific case of $\Out(F_3)$, \cref{thm:intro_atiyah} can also be deduced from work of McCool \cite{McCool}, who proved that it is virtually residually torsion-free nilpotent.

We already mentioned that for a given torsion-free group $G$, the Strong Atiyah Conjecture over $\K$ implies Kaplansky's Zero Divisor Conjecture for $\K[G]$. In this article we will consider a strengthening of the Zero Divisor Conjecture, which is also wide open.

\begin{conj}\label{conj:div_ring}
    If $G$ is a torsion-free group and $\K$ is a field, then $\K[G]$ embeds into a division ring.
\end{conj}

For $\K \subseteq \C$, Linnell showed that the Strong Atiyah Conjecture over $\K$ for a torsion-free group $G$ is equivalent to the division closure of $\K[G]$ in its algebra of affiliated operators being a division ring \cite{LinnellDivRings93}. Hence, \cref{conj:div_ring} can be thought of as a positive characteristic analogue of the Atiyah Conjecture. \cref{conj:div_ring} has been established for many classes of groups, including residually  \{torsion-free elementary amenable groups\} \cite{KrophollerLinnellMoody_Ore} and consequently torsion-free virtually compact special groups \cite{Schreve_AtiyahVCS}, torsion-free $3$-manifold groups \cite{fishersanchez_divrings}, and bi-orderable groups \cite{Malcev_series,Neumann_series}. A notable open special case of \cref{conj:div_ring} known as \emph{Malcev's Problem} asks about left orderable groups; this is open even for locally indicable groups.

Since $\GL_m(\Z)$ has torsion, its group algebra cannot be a domain. However, it is known from work of Linnell and Farkas \cite{FarkasLinnell2006}, building on the classical results of Lazard \cite{Lazard1954}, that if $C(p) \leqslant \GL_n(\Z)$  denotes the level $p$ (if $p > 2$) or level $p^2$ (if $p = 2$) principle congruence subgroup, then $\K[C(p)]$ admits an embedding into a division ring for any field of characteristic $p$ or $0$. Thus, $\GL_m(\Z)$ virtually satisfies \cref{conj:div_ring} over any field $\K$. Using a similar method of proof as in \cref{thm:intro_atiyah} we extend these embeddings to certain outer automorphism groups, and in particular exhibit torsion-free finite-index subgroups of $\Out(F_n)$ and $\Mod(\Sigma_g)$ satisfying the Zero Divisor Conjecture.

\begin{thm}\label{thm:intro_div_ring}
    Let $A_\Gamma$ be the RAAG on the finite simplicial graph $\Gamma$, let $\Sigma_{g,b,p}$ be a compact surface of genus $g$ with $b \geqslant 0$ boundary components and $p \geqslant 0$ marked points, and let $\K$ be a field. There are torsion-free finite-index subgroups $G \leqslant \Out(A_\Gamma)$ and $H \leqslant \Mod(\Sigma_{g,b,p})$ such that $\K[G]$ and $\K[H]$ embed into division rings. In particular, $\K[G]$ and $\K[H]$ satisfy the Zero Divisor Conjecture.
\end{thm}

\subsection{Proof outline}

We give a short overview of the proof of the Atiyah and Zero Divisor Conjectures for $\Out(A_\Gamma)$, where $A_\Gamma$ is the RAAG on the simplicial graph $\Gamma$. The main idea is to leverage Farkas and Linnell's proof of the Atiyah and Zero Divisor Conjectures for torsion-free compact $p$-adic analytic groups \cite{FarkasLinnell2006} and a general form of the L\"uck Approximation Theorem due to Jaikin-Zapirain \cite{JaikinZapirain_basechange} in order to prove the Atiyah and Zero Divisor Conjectures for the larger group $\Out(\mathbf A_\Gamma)$ (or a suitable torsion-free subgroup of finite index), where $\mathbf A_\Gamma$ denotes the pro-$p$ completion of $A_\Gamma$. More precisely we have the following result, which strengthens \cref{thm:intro_atiyah,thm:intro_div_ring}.

\begin{thm}\label{thm:intro_main_pro_p_case}
    Let $\K$ be a field of characteristic $p$. The group $\Out(\mathbf A_\Gamma)$ has an open torsion-free pro-$p$ subgroup $\mathbf G$ such that 
    \begin{enumerate}
        \item $\mathbf G$ satisfies the Strong Atiyah Conjecture over $\C$ and
        \item the completed group algebra $\K\llbracket \mathbf G \rrbracket$ embeds into a division ring.
    \end{enumerate}
\end{thm}

We formulated the Strong Atiyah Conjecture for countable groups, but $\mathbf A_\Gamma$ is uncountable if it is non-trivial. We will say that an uncountable group satisfies the Strong Atiyah Conjecture over $\K$ if all its finitely generated subgroups do.

We establish \cref{thm:intro_main_pro_p_case} as follows. In \cref{thm:propMagnus}, we show that $\mathbf A_\Gamma$ admits a Magnus representation as a multiplicative group of power series in partially commuting variables with $\Z_p$ coefficients. This is the key tool we use to show that the groups
\[
    \mathcal T^{(i)}(\mathbf A_\Gamma) = \ker\big( \Out(\mathbf A_\Gamma) \to \Out(\mathbf A_\Gamma/\gamma_{i+1} (\mathbf A_\Gamma))\big)
\]
form a residual chain inside $\Out(\mathbf A_\Gamma)$ and that the consecutive quotients 
\[
    \mathcal T^{(i)}(\mathbf A_\Gamma)/\mathcal T^{(i+1)}(\mathbf A_\Gamma)
\]
are finitely generated torsion-free Abelian pro-$p$ groups (see \cref{thm:prop_RTFN_Torelli}). In particular, we have the following result which is the pro-$p$ version of the analogous result in the discrete case due independently to Toinet \cite[Theorem 7.14]{Toinet2013} and Wade \cite[Theorem 4.9]{Wade_JohnsonHomsHigherRank}.

\begin{thm}\label{thm:TorelliRTFN}
    For any pro-$p$ RAAG $\mathbf A_\Gamma$, the Torelli subgroup $\mathcal T(\mathbf A_\Gamma)$ of $\Out(\mathbf A_\Gamma)$ is residually torsion-free nilpotent.
\end{thm}

The groups $\gamma_i(\mathbf A_\Gamma)$ are the terms of the lower central series of $\mathbf A_\Gamma$ and the subgroups $\mathcal T^{(i)}(\mathbf A_\Gamma)$ form the \emph{Andreadakis--Johnson filtration} of $\Out(\mathbf A_\Gamma)$. In the discrete case, the fact that it is residual and has finite generated Abelian consecutive quotients was established by Wade in \cite{Wade_JohnsonHomsHigherRank} (see also \cite{Andreadakis1965} for the free group case, and \cite{Asada1995} and \cite{Johnson_survey} for the surface group case). The idea of studying a group using this filtration was formalised and exploited by Bass and Lubotzky in their work on linear central filtrations \cite{BassLubotzky1994}. Asada developed the idea for pro-$p$ groups in \cite{Asada1995}, and we use his results crucially in proving \cref{thm:TorelliRTFN}.

Thus, $\Out(\mathbf A_\Gamma)$ is approximated by the groups $\Out(\mathbf A_\Gamma)/\mathcal T^{(i)}(\mathbf A_\Gamma)$, each of which is compact $p$-adic analytic (using the fact that extensions of $p$-adic analytic groups are $p$-adic analytic). \cref{thm:intro_main_pro_p_case} can then be deduced from the Atiyah and Zero Divisor Conjectures for torsion-free compact $p$-adic analytic group \cite{FarkasLinnell2006} together with a sofic approximation theorem for von Neumann rank functions \cite{JaikinZapirain_basechange}. In the same manner, we will obtain similar results for automorphism groups of many residually torsion-free nilpotent groups (see \cref{cor:aut_atiyah}).

To obtain our main results in the discrete case, we only need for the natural map $\Out(A_\Gamma) \to \Out(\mathbf A_\Gamma)$ to be injective, which is equivalent to the claim that any automorphism of $A_\Gamma$ which induces an inner automorphism of all finite $p$-quotients of $A_\Gamma$ is in fact inner. For RAAGs, it follows from conjugacy $p$-separability \cite[Theorem 6.15]{Toinet2013} and the fact that their pointwise inner automorphisms are inner \cite[Proposition 6.9]{Minasyan2012}. It appears that the property that the natural map $\Out(G) \rightarrow \Out(\mathbf G)$ being injective (for $G$ a residually $p$ group and $\mathbf G$ its pro-$p$ completion) was first explicitly considered by Segal in \cite{Segal90}, where he gave examples showing that it fails for certain finitely generated torsion-free nilpotent groups. In an appendix, the first author establishes this injectivity for free nilpotent groups, which form a special class of torsion-free nilpotent groups.

\subsection{Organisation of the paper}

In \cref{sec:prelims} we recall some preliminary notions, in particular on the Atiyah Conjecture and the theory of $p$-adic analytic groups. In \cref{sec:magnus_RAAG}, we show that pro-$p$ completions of RAAGs admit a version of the Magnus embedding into a ring of power series in partially commuting variables and coefficients in $\Z_p$, generalising the embedding of Lazard for a free pro-$p$ group \cite{Lazard1954}. This is then used to prove that the Torelli group of a pro-$p$ RAAG is residually torsion-free nilpotent in \cref{thm:prop_RTFN_Torelli}. In \cref{sec:main_results}, we prove our main results, namely Theorems \ref{thm:intro_atiyah}, \ref{thm:intro_div_ring}, and \ref{thm:intro_main_pro_p_case}. Finally, in an appendix the first author proves that $p$-inner automorphisms of free nilpotent groups are inner.

\subsection{Acknowledgments}

The authors are grateful to Andrei Jaikin-Zapirain for many useful conversations and suggestions that led to the completion of this work, and thank Ric Wade for comments on an earlier version of this article. The first author is grateful to Dan Segal for a helpful correspondence. The first author is supported by  the grants CEX2023-001347-S and EUR2025-164928 of the Ministry of Science, Innovation, and Universities of Spain. The second author is supported by funding from the European Union (ERC, SATURN, 101076148) and the Deutsche Forschungsgemeinschaft (EXC-2047/1 - 390685813).

\section{Preliminaries}\label{sec:prelims}

\subsection{Notation and conventions}

We fix the notation we will use throughout the paper. Let $G$ be a group and let $a,b \in G$. We put $a^b := b\inv a b$ (the \emph{conjugate} of $a$ by $b$) and $[a,b] := a\inv b\inv a b$ (the \emph{commutator} of $a$ and $b$). If $A$ and $B$ are subgroups of $G$, then we write $[A,B]$ for the subgroup generated by the elements $[a,b]$ for $a \in A$ and $b \in B$.

The \emph{lower central series} $(\gamma_n(G))_{n\in\N}$ of $G$ is defined recursively as follows: 
\[
    \gamma_1(G) := G, \quad \gamma_n(G) := [G,\gamma_{n-1}(G)] \quad \text{for} \ n > 1.
\]
Note that the subgroup $\gamma_n(G)$ is characteristic in $G$ for all $n \in \N$. The group $G$ is \emph{nilpotent} if $\gamma_n(G) = \{1\}$ for some $n \in \N$, and it is \emph{residually nilpotent} if $\bigcap_{n\in\N} \gamma_n(G) = \{1\}$.

Rings are always assumed to be unital and associative, and ring homomorphisms preserve the unit. We will occasionally consider Lie algebras, which are not associative, but we do not consider them as rings, instead forming a different category. If $R$ is a ring and $x,y \in R$, we write $[x,y] := xy - yx$.

For an arbitrary but fixed prime $p$, if a group is denoted by a letter, we will usually denote the pro-$p$ completion of that group by the same letter in bold font. For instance, the pro-$p$ completions of $G$ and $H$ are denoted by $\mathbf G$ and $\mathbf H$ respectively. In many instances groups will have subscripts attached to them. For example if $A_\Gamma$ is the right-angled Artin group on $\Gamma$, then we will denote its pro-$p$ completion by $\mathbf A_\Gamma$. We will use $\Z_p$ and $\Q_p$ to denote the $p$-adic integers and $p$-adic rational numbers respectively.

\subsection{The Andreadakis--Johnson filtration}

Given a group $G$, we now describe filtrations of $\Aut(G)$ and $\Out(G)$ originally studied by Andreadakis and Johnson \cite{Andreadakis1965, Johnson_survey}. For these filtrations to give useful information about the outer automorphism group, a minimum requirement is that $G$ be residually nilpotent. However, in practice even stronger hypotheses on $G$ are needed (see \cref{thm:Asada} below). If $G$ is finitely generated pro-$p$ group, then in all the definitions given below the reader should keep in mind that $\Aut(G)$ denotes the group of all continuous automorphisms of $G$.

Since the subgroups $\gamma_n(G)$ are all characteristic in $G$, there are natural homomorphisms
\[
    \varphi_i \colon \Aut(G) \to \Aut(G/\gamma_n(G)) \quad \text{and} \quad \ \overline{\varphi_i} \colon \Out(G) \to \Out(G/\gamma_n(G))
\]
for all $n \in \N$. The $n$th terms of the \emph{Andreadakis--Johnson filtration} of $\Aut(G)$ and $\Out(G)$ are the subgroups 
\[
    T^{(n)}(G) := \ker(\varphi_{n+1}) \quad \text{and} \quad \mathcal T^{(n)}(G) := \ker(\overline{\varphi_{n+1}}),
\]
respectively. The terms $T^{(1)}(G)$ and $\mathcal T^{(1)}(G)$ correspond to the kernels of the action on the Abelianisation of $G$, and are called the \emph{Torelli groups} of $\Aut(G)$ and $\Out(G)$, respectively. We will write $T(G) = T^{(1)}(G)$ and $\mathcal T(G) = \mathcal T^{(1)}(G)$.

We will also consider the modulo $p$ version of the Torelli group:
\[
    T_p(G) := \ker\left(\Aut(G) \to \Aut\big(G/\gamma_2(G)G^{p^\varepsilon}\big)\right)
\]
and 
\[
    \mathcal T_p(G) := \ker\left(\Out(G) \to \Out\big(G/\gamma_2(G)G^{p^\varepsilon}\big)\right)
\]
where $\varepsilon = 1$ if $p > 2$ and $\varepsilon = 2$ if $p = 2$.

It will be important that the Johnson filtration is residual and has finitely generated Abelian consecutive quotients. To ensure this we will use the following criterion of Asada.

\begin{thm}[{\cite{Asada1995}}]\label{thm:Asada}
    Let $\mathbf G$ be a finitely generated pro-$p$ group. If
    \begin{enumerate}
        \item $Z(\gr(\mathbf G)) = \{0\}$,
        \item $Z(\gr(\mathbf G) \otimes_{\Z_p} \F_p) = \{0\}$, and
        \item $\gr(\mathbf G)$ is free as a $\Z_p$-module,
    \end{enumerate}
    then $\bigcap_{n\in \N} \mathcal T^{(n)}(\mathbf G) = \{1\}$ and $\mathcal T^{(n)}(\mathbf G)/\mathcal T^{(n+1)}(\mathbf G)$ is a finitely generated free $\Z_p$-module for all $n \in \N$.
\end{thm}

The object $\gr(\mathbf G)$ appearing in \cref{thm:Asada} is the graded Lie algebra associated to the lower central series of $\mathbf G$. Its definition is recalled more carefully in \cref{def:Lie_algebras}.

\subsection{\texorpdfstring{$L^2$}{L²}-invariants and the Atiyah Conjecture}\label{subsec:atiyah}

The standard reference for this material is L\"uck's book \cite{Luck02}, to which we refer the reader for details not covered here. Let $G$ be a countable group and let
\[
    \ell^2(G) = \left\{\sum_{g \in G} \alpha_g g \ : \ \alpha_g \in \C, \ \sum_{g \in G} |\alpha_g|^2 < \infty \right\}
\]
be the space of square summable series of group elements with complex coefficients. Then $\ell^2(G)$ is naturally a Hilbert space with inner product given by 
\[
    \left\langle \sum_{g \in G} \alpha_g g, \sum_{g \in G} \beta_g g \right\rangle := \sum_{g \in G} \alpha_b \overline{\beta_g}.
\]

The left regular action of $G$ on $\ell^2(G)$ extends to a component-wise action of $G$ on $\ell^2(G)^n$. For any $n \in \N$, a closed $G$-invariant subspace $V$ of $\ell^2(G)^n$ is called a \emph{Hilbert $G$-module}. If $\pi_V \colon \ell^2(G)^n \to \ell^2(G)^n$ denotes the orthogonal projection onto $V$, then the \emph{von Neumann dimension} of $V$ is given by the trace of $\pi_V$, defined by
\[
    \dim_{\mathcal N(G)} (V) := \sum_{i=1}^n \langle 1_i, \pi_V(1_i) \rangle,
\]
where $1_i$ denotes the vector of $\ell^2(G)$ that consists of the identity element $1$ in the $i$th component. It is a standard exercise to verify that $\dim_{\mathcal N(G)}(V)$ does not depend on the embedding of $V$ into some $\ell^2(G)^m$ as a closed subspace.

An $m \times n$ matrix $A$ over $\C[G]$ induces a linear map $\varphi_A \colon \ell^2(G)^m \to \ell^2(G)^n$ via right multiplication. Hence $\im \varphi_A$ is a (possibly non-closed) $G$-invariant subspace of $\ell^2(G)^n$. The \emph{von Neumann rank} of $A$ is
\[
    \rk_{\mathcal N(G)} (A) := \dim_{\mathcal N(G)} \left( \overline{\im \varphi_A} \right),
\]
where $\overline{\im \varphi_A}$ denotes the closure of $\im \varphi_A$. Everything is now in place for a complete statement of the Strong Atiyah Conjecture given in the introduction (see \cref{conj:SAC}): if $G$ is a group with an upper bound on the orders of its finite subgroups, then the Strong Atiyah Conjecture over a subfield $\K \subseteq \C$ predicts that $\rk_{\mathcal N(G)} (A) \in \frac{1}{\mathrm{lcm}(G)} \Z$ for all finite matrices $A$ with coefficients in $\K[G]$.

We now recall the connection between the Atiyah and Zero Divisor Conjectures. The \emph{von Neumann algebra} of $G$ is the algebra of bounded operators acting on $\ell^2(G)$ and commuting with the left $G$-action:
\[
    \mathcal N(G) := \left\{ A \in \mathcal B(\ell^2(G)) : (gx)A = g(xA) \ \text{for all} \ x \in \ell^2(G) \ \text{and} \ g \in G \right\}.
\]
It is immediate from the definition that $\C[G] \subseteq \mathcal N(G)$.
By a result of Berberian \cite{Berberian_vonNeumannOre}, $\mathcal N(G)$ satisfies the left and right Ore conditions with respect to its set of non-(zero divisors). The (left) Ore localisation of $\mathcal N(G)$ at the set of non-(zero divisors) is called the \emph{algebra of affiliated operators} of $G$ and is denoted by $\mathcal U(G)$. In summary, we have the following subring inclusions: $\C[G] \subseteq \mathcal N(G) \subseteq \mathcal U(G)$.

If $R$ is a subring of $S$, then the \emph{division closure} of $R$ in $S$ is the smallest subring $\mathcal D(R \subseteq S)$ of $S$ containing $R$ such that if $x$ is a unit of $S$ and $x \in \mathcal D(R \subseteq S)$, then $x\inv \in \mathcal D(R \subseteq S)$. The connection between \cref{conj:SAC,conj:div_ring} is given by the following fundamental result of Linnell.

\begin{thm}[{\cite{LinnellDivRings93}}]\label{thm:Linnell}
    If $G$ is a torsion-free group, then $G$ satisfies the Strong Atiyah Conjecture over $\K \subseteq \C$ if and only if $\mathcal D(\K[G] \subseteq \mathcal U(G))$ is a division ring. 
\end{thm}

In particular, if $G$ is torsion-free and satisfies the Strong Atiyah Conjecture over $\K \subseteq \C$, then $\K[G]$ satisfies the Zero Divisor Conjecture (and if $\K = \C$, then $\mathbb L[G]$ satisfies the Zero Divisor Conjecture for any field $\mathbb L$ of characteristic zero).

While it is not known whether the Atiyah Conjecture passes to subgroups in general, this is the case for subgroups of torsion-free groups.

\begin{lem}[{\cite[Lemma 10.4]{Luck02}}]\label{lem:SAC_subgroups}
    Let $G$ be a torsion-free group satisfying the Strong Atiyah Conjecture over $\K \subseteq \C$. If $H \leqslant G$, then $H$ satisfies the Strong Atiyah Conjecture over $\K$. 
\end{lem}

Our proof of the Weak Atiyah Conjecture for various outer automorphism groups goes via proving the Strong Atiyah Conjecture for a finite-index subgroup. This relies on the following result.

\begin{lem}\label{lem:weak_atiyah_commensurability}
    If $G$ has a finite-index subgroup $H$ satisfying the Weak Atiyah Conjecture over a field $K \subseteq \C$, then $G$ satisfies the Weak Atiyah Conjecture over $K$.
\end{lem}
\begin{proof}
    This follows immediately from the fact that $\rk_{\mathcal N(G)} = \frac{1}{[G:H]} \rk_{\mathcal N(H)}$ \cite[Theorem 6.54 (6)]{Luck02}.
\end{proof}

One of the main results we will use is Jaikin-Zapirain's sofic approximation theorem for von Neumann rank functions. The definition of soficity will not concern us here, as it will suffice to know that residually finite groups are sofic and the property passes to subgroups. The most general version of the following theorem is for convergence in the space of marked groups, but we will only need it in the form stated below. We say that a chain of subgroup $G = G_1 \geqslant G_2 \geqslant \dots$ is a \emph{residual normal chain} if $\bigcap_{i \in \N} G_i = \{1\}$ and $G_i$ is normal in $G$ for all $i \in \N$.

\begin{thm}[{\cite{JaikinZapirain_basechange}}] \label{thm:sofic_approx}
    Let $G$ be a finitely generated group and let 
    \[
        G = G_1 \geqslant G_2 \geqslant \dots
    \]
    be a residual normal chain such that every quotient $G/G_k$ is sofic. If $A$ is a finite matrix $A$ over $\C[G]$, then
    \[
        \rk_{\mathcal N(G)} (A) = \lim_{k \to \infty} \rk_{\mathcal N(G/G_k)}(A_k),
    \]
    where $A_k$ is the matrix over $\C[G/G_k]$ obtained from $A$ by applying the quotient map $\C[G] \to \C[G/G_k]$ to all its entries.
\end{thm}

Recall that, for a not necessarily countable group $G$ with a bound on the order of its finite subgroups, we will say that $G$ satisfies the Strong Atiyah Conjecture over $\C$ if all of its finitely generated subgroups satisfy it.

\begin{cor}\label{cor:approx_Atiyah}
    Let $G$ be a group and let $G = G_1 \geqslant G_2 \geqslant \dots$ be a residual normal chain such that every quotient $G/G_k$ is torsion-free, sofic, and satisfies the Strong Atiyah Conjecture over $\C$. Then $G$ satisfies the Strong Atiyah Conjecture over $\C$.
\end{cor}
\begin{proof}
    If $H \leqslant G$ is a subgroup, then the subgroups $H_n := H \cap G_n$ define a residual normal chain in $H$ such that $H/H_n$ is torsion-free, sofic, and moreover $H/H_n$ satisfies the Strong Atiyah Conjecture over $\C$ by \cref{lem:SAC_subgroups}. We may therefore assume that $G$ is finitely generated. For every finite matrix $M$ over $\C[G]$, we have that $\rk_{\mathcal N(G)}(M)$ is a limit of integers by \cref{thm:sofic_approx}. Hence, $\rk_{\mathcal N(G)}(M) \in \Z$, and therefore $G$ satisfies the Strong Atiyah Conjecture over $\C$. \qedhere
\end{proof}

We will also use the following fact about division ring embeddability, which is of a similar nature, but has a much easier proof. We only briefly sketch its proof, since it is well known.

\begin{lem}\label{lem:approx_div_rings}
    Let $R$ be a ring and let $G$ be a group with a residual normal chain $(G_n)_{n\in\N}$. If $R[G/G_n]$ embeds into a division ring for all $n \in \N$, then $R[G]$ embeds into a division ring.
\end{lem}
\begin{proof}[Proof (sketch)]
    Let $\omega$ be a non-principal ultrafilter on $\N$. Suppose that $R[G/G_k]$ embeds into the division ring $\mathcal D_k$ for each $k$. There are ring homomorphisms
    \[
        R[G] \to R[G/G_k] \to \mathcal D_k
    \]
    for each $k$, and together they induce an embedding $R[G] \hookrightarrow \prod_\omega \mathcal D_k$. But the ultraproduct of division rings over a non-principal ultrafilter is again a division ring. \qedhere
\end{proof}

Finally, to obtain division ring embeddings for arbitrary fields of characteristic zero, we will use the following result. The proof is the same as that of \cite[Corollary 6.7]{JaikinLopezStrongAtiyah2020}.

\begin{lem}\label{lem:sofic_SAC_all_fields}
    If $G$ is a torsion-free sofic group satisfying the Strong Atiyah Conjecture over $\C$, then $\K[G]$ embeds into a division ring for all fields of characteristic zero. 
\end{lem}
\begin{proof}
    Let $A$ be a finite matrix over $\K[G]$. Let $\K_0$ be a finitely generated subfield of $\K$ containing all the coefficients appearing in the entries of $A$ and fix an embedding $\varphi \colon \K_0 \hookrightarrow \C$. Let $A^\varphi$ be the matrix obtained from $A$ by applying $\varphi$ to all the coefficients appearing in its entries. Put
    \[
        \rk(A) := \rk_{\mathcal N(G)}(A^\varphi),
    \]
    and observe that $\rk(A)$ is independent of the choice of $\varphi$ by the fact that sofic groups satisfy the Independence Conjecture \cite[Corollary 1.7]{JaikinZapirain_basechange}. Since $G$ is torsion-free and satisfies the Strong Atiyah Conjecture over $\C$, we have that $\rk(A) \in \Z$ for all matrices $A$ over $\K[G]$. Thus $\K[G]$ admits an integer-valued Sylvester matrix rank function. A theorem of Malcolmson \cite[Theorem 1]{Malcolmson_SkewFields} then implies that there is a ring homomorphism $\psi\colon \K[G] \to \mathcal D$ where $\mathcal D$ is a division ring and $\rk(A)$ is the rank of the matrix obtained from $A$ by applying $\psi$ to all its entries. If $\psi(x) = 0$ for some $x \in \K[G]$, then $\rk(x) = 0$. But by the definition of $\rk$ implies this is only possible if $x = 0$. Hence, $\psi$ is injective. \qedhere
\end{proof}

\subsection{\texorpdfstring{$p$}{p}-adic analytic groups}

Let $p$ be a prime. A $p$-adic analytic group is a non-Archimedean analogue of a Lie group. Roughly, these are topological groups that admit a system of charts over $\Z_p^n$ for some $n \in \N$ such that the transition maps between the charts are analytic in a suitable sense, and moreover such that the group operations of multiplication and inversion are analytic maps. We will not give the precise definition here since it is rather involved and we will not need it. Instead we refer the reader to the foundational work of Lazard \cite{Lazard1965}, where many properties analogous to those of Lie groups were established, or to the book \cite{DixonEtAl_analyticPropBook} for further details on the subject. We will take the following theorem as a definition of compact $p$-adic analytic groups.

\begin{thm}[{\cite[Theorem 7.9 and Corollary 8.33]{DixonEtAl_analyticPropBook}}]
    A compact group $G$ is $p$-adic analytic if and only if it embeds as a closed subgroup of $\GL_n(\Z_p)$ for some $n \in \N$.
\end{thm}

Let $G$ be a topological group. We use the notation $U \trianglelefteqslant_o G$ to indicate that $U$ is an open normal subgroup of $G$. Given a ring $R$, the associated \emph{completed group algebra} of $G$ with coefficients in $R$ is
\[
    R\llbracket G \rrbracket = \varprojlim_{U \trianglelefteqslant_o G} R[G/U].
\]
In \cite[Proposition V.2.2.4]{Lazard1954}, Lazard shows that $\Z_p\llbracket G \rrbracket$ is Noetherian whenever $G$ is compact $p$-adic analytic. Moreover, $\Z_p\llbracket G \rrbracket$ is also a domain whenever $G$ is additionally torsion-free by a theorem of Neumann \cite[Theorem 1]{Neumann_completedGpAlgDomain}. Generalising these results, we have the following theorem of Farkas and Linnell.

\begin{thm}[{\cite[Theorem 6.1]{FarkasLinnell2006}}]\label{thm:p_adic_completed_domain}
    If $G$ is a torsion-free compact $p$-adic analytic group and $\K$ is a field of characteristic $p$, then the completed group algebra $\K\llbracket G\rrbracket$ is a Noetherian domain, and in particular embeds into a division ring.
\end{thm}

The last part of the theorem follows from the standard fact that Noetherian domains are Ore domains and thus embed into a division ring of fractions \cite[p.\ 47]{McConnellRobsonNNR}.

Farkas--Linnell \cite[Theorem 1.1]{FarkasLinnell2006} show that torsion-free compact $p$-adic analytic groups satisfy the Strong Atiyah Conjecture over $\overline{\Q}$, the algebraic closure of $\Q$, and Jaikin-Zapirain \cite[Theorem 1.1]{JaikinZapirain_basechange} proves that if a sofic group satisfies the Strong Atiyah Conjecture over $\overline{\Q}$, then it also satisfies it over $\C$. Consequently we have that following result. We remark that this could also be deduced from an approximation result of Harris \cite[Lemma 1.10.1]{Harris_padicdescent_1979}.

\begin{thm} \label{thm:p-adic-atiyah}
    If $G$ is a torsion-free compact $p$-adic analytic group, then $G$ satisfies the Strong Atiyah Conjecture over $\C$.
\end{thm}

Finally, the fact that the class of $p$-adic analytic groups is closed under extensions will be crucial in the proofs of the main results.

\begin{lem}[{\label{lem:p_adic_extension}\cite[Theorems 9.6 and 9.7]{DixonEtAl_analyticPropBook}}]
    Let $G$ be a Hausdorff topological group and $N \trianglelefteqslant G$ be a closed normal subgroup. Then $G$ is $p$-adic analytic if and only if $N$ and $G/N$, with the induced and quotient topologies respectively, are $p$-adic analytic.
\end{lem}

\section{The Magnus embedding of a pro-\texorpdfstring{$p$}{p} RAAG} \label{sec:magnus_RAAG}

We begin this section by revisiting the Magnus embedding of a right-angled Artin group. The idea of embedding a group into the unit group of a ring of power series goes back to Magnus, who showed that free groups admit such embeddings \cite[Satz I]{Magnus_freegrpsRTFN}. The Magnus embedding is a powerful tool to study the lower central series of a free group; in particular, one can use it to show that free groups are residually torsion-free nilpotent and bi-orderable. Magnus's embedding was extended by Lazard \cite{Lazard1954}, who showed that free pro-$p$ groups (i.e.\ the pro-$p$ completion of a discrete free group) admit similar embedding into power series rings with $\Z_p$ coefficients. In \cite{DuchampKrob_RAAGsRTFN}, Duchamp and Krob show that RAAGs admit Magnus embeddings. We extend this result to show that the pro-$p$ completion of a RAAG has a Magnus embedding over $\Z_p$. Once this is achieved we will use it to study the residual properties of the Torelli group of $\Out(\mathbf A_\Gamma)$ for a finitely generated pro-$p$ RAAG $\mathbf A_\Gamma$.

Let $R$ be a unital ring and let $X$ be an abstract set. The \emph{free (associative, unital) $R$-algebra} on $X$ is denoted by $R\langle X \rangle$; it is the free object in the category of unital $R$-algebras on $|X|$ generators. We think of $R\langle X\rangle$ as a polynomial ring in the non-commuting variables of $X$, with no imposed relations between them. Let $\Gamma$ be a simplicial graph and let $V(\Gamma)$ and $E(\Gamma)$ be its vertex and edge sets, respectively. Define a two-sided ideal of $R\langle V(\Gamma)\rangle$ by
\[
    c_\Gamma = \left( [u,v] : \{u,v\} \in E(\Gamma) \right),
\]
where $[u,v] = uv - vu$. The \emph{free partially commutative $R$-algebra} with commutation relations induced by $\Gamma$ is $R\langle \Gamma \rangle = R\langle V(\Gamma)\rangle / c_\Gamma$. So $R\langle \Gamma \rangle$ is a polynomial ring in the partially commuting variables $V(\Gamma)$.

\begin{defn}[The Magnus ring of a graph]
    Let $\omega_{R\langle\Gamma\rangle} \subseteq R\langle\Gamma\rangle$ be the two-sided ideal generated by $X = V(\Gamma)$; we will occasionally write $\omega$ for $\omega_{R\langle\Gamma\rangle}$ when no confusion can arise. The \emph{Magnus ring} associated to the graph $\Gamma$ with coefficients in $R$ is
    \[
        R\llangle \Gamma \rrangle = \varprojlim_{n \in \N} R\langle\Gamma\rangle/\omega^n.
    \]
\end{defn}

Note that $V(\Gamma)$ still embeds into $R\llangle \Gamma\rrangle$, and the Magnus ring $R\llangle\Gamma\rrangle$ can be thought of as the ring of infinite formal power series in the partially commuting elements of $V(\Gamma)$, where elements of $V(\Gamma)$ commute if and only if they form an edge of $\Gamma$. We also define $\omega_{R\llangle\Gamma\rrangle}$ to be the two-sided ideal of $R\llangle \Gamma\rrangle$ generated by $V(\Gamma)$.

The Magnus ring $R\llangle \Gamma\rrangle$ has a rich unit group, which in particular contains all elements in $1 + \omega_{R\llangle\Gamma\rrangle}$. Indeed, if $f \in \omega_{R\llangle\Gamma\rrangle}$, then the inverse of $1 + f$ is given by $\sum_{i=0}^\infty (-f)^i$. Hence, $1 + \omega_{R\llangle\Gamma\rrangle}$ is actually a subgroup of the unit group $R\llangle \Gamma\rrangle^\times$. 

\begin{defn}[Right-angled Artin groups]
    Let $\Gamma$ be a simplicial graph and let $S$ be a set in one-to-one correspondence with $V(\Gamma)$ via the map $v \mapsto s_v$. The \emph{free partially commutative group} or \emph{right-angled Artin group} (also abbreviated as \emph{RAAG}) on $\Gamma$ is the group $A_\Gamma$ defined by the presentation
    \[
        \left\langle S \mid [s_u,s_v] \ \text{if and only if} \ \{u,v\} \in E(\Gamma) \right\rangle.
    \]
\end{defn}

By the following result of Duchamp and Krob, $1 + \omega_{R\llangle\Gamma\rrangle}$ contains a copy of $A_\Gamma$.

\begin{thm}[{\cite[Theorems 1.2 and 2.2]{DuchampKrob_RAAGsRTFN}}]\label{thm:DuchampKrob}
    Let $A_\Gamma$ be the RAAG on $\Gamma$ and let $R$ be a ring.
    \begin{enumerate}
        \item There is a homomorphism
        \[
            \mu_R \colon A_\Gamma \to 1 + \omega_{R\llangle\Gamma\rrangle}
        \]
        determined by $\mu_R(s_v) = 1 + v$.
        \item The map $\mu_R$ is injective.
        \item For all $n \in \N$, we have $\mu_\Z\inv(1 + \omega_{\Z\llangle\Gamma\rrangle}^n) = \gamma_n(A_\Gamma)$.
    \end{enumerate}
\end{thm}

When the ring $R$ is understood, we will write $\mu = \mu_R$. As a consequence, RAAGs are residually torsion-free nilpotent, and in particular they are residually $p$ for all primes $p$. Hence, $A_\Gamma$ embeds into its pro-$p$ completion $\mathbf A_\Gamma$. 

Let $R\langle\Gamma\rangle_0 = R\langle \Gamma\rangle$ and by induction, we let $R\langle\Gamma\rangle_{i+1}$ be the $R$-submodule of $R\langle\Gamma\rangle$ generated by $R\langle\Gamma\rangle_i \cdot V(\Gamma)$ and $V(\Gamma) \cdot R\langle\Gamma\rangle_i$. Then 
\[
    R\langle\Gamma\rangle = \bigoplus_{i \geqslant 0} R\langle\Gamma\rangle_i
\]
as an $R$-module. We call elements $f$ of $R\langle\Gamma\rangle$ lying in some $R\langle\Gamma\rangle_i$ \emph{homogeneous of degree $i$}. Any element $f \in R\llangle\Gamma\rrangle$ can be uniquely expressed as a series $\sum_{i \geqslant 0} f_i$ such that $f_i \in R\langle\Gamma\rangle_i$. We call $f_i$ the \emph{degree $i$ component} of $f$.

\begin{defn}[Lie algebras]\label{def:Lie_algebras}
    \begin{enumerate}
        \item Given an associative ring $R$, one always has a Lie algebra $\mathfrak{Lie}(R)$ whose underlying Abelian group coincides with that of $R$ and whose Lie bracket is defined by $[r,s] = rs - sr$. The associativity of $R$ ensures that $\mathfrak{Lie}(R)$ is indeed a Lie algebra.
        \item The \emph{graded Lie algebra} of a group $G$ is
        \[
            \mathfrak{gr}(G) = \bigoplus_{i \in \N} \gamma_i(G)/\gamma_{i+1}(G).
        \]
        Each factor $\gamma_i(G)/\gamma_{i+1}(G)$ is an Abelian group, and their Abelian group structures induce the addition on $\gr(G)$. The Lie bracket is induced by
        \[
            [g \gamma_{i+1}(G), h \gamma_{j+1}(G)] = [g,h] \gamma_{i+j+1}(G),
        \]
        where $g \in \gamma_i(G)$ and $h \in \gamma_j(G)$, and the commutator $[g,h]$ on the right-hand side is the group theoretic commutator $g\inv h\inv g h$.
    \end{enumerate}
\end{defn}

For each $i \in \N$, there is a well-defined map
\[
    \nu_i \colon \gamma_i(A_\Gamma)/\gamma_{i+1}(A_\Gamma) \to \Z\langle\Gamma\rangle_i, \quad g\gamma_{i+1}(A_\Gamma) \mapsto \mu(g)_i
\]
of Abelian groups, and together these induce a map
\[
    \nu = \bigoplus_{i \in \N} \nu_i \colon \mathfrak{gr}(A_\Gamma) \to \mathfrak{Lie}(\Z\langle\Gamma\rangle)
\]
We record two facts about the relationship between $\mathfrak{gr}(A_\Gamma)$ and $\mathfrak{Lie}(R\langle \Gamma\rangle)$ proved by Wade in \cite[Theorem 3.5 and Corollary 3.7]{Wade_JohnsonHomsHigherRank}.

\begin{thm}\label{thm:LieAlg_multiples}
    \begin{enumerate}
        \item The map $\nu$ is an injective homomorphism of Lie algebras, i.e.\ $\nu$ preserves the Lie bracket.
        \item\label{item:multiples} If $n \in \Z$ and $x \in \mathfrak{gr}(A_\Gamma)$, then $\nu(x) \in n \mathfrak{Lie}(\Z\langle\Gamma\rangle)$ if and only if $x \in n\mathfrak{gr}(A_\Gamma)$.
    \end{enumerate}
\end{thm}

We rephrase \ref{item:multiples} in a form we will use later.

\begin{cor}\label{cor:powers_mod_LCS}
    Let $g \in \gamma_i(A_\Gamma)$ for some $i \geqslant 1$. If $\mu(g)_i$ is a multiple of $n$, then there is an element $h \in \gamma_i(A_\Gamma)$ such that $g\gamma_{i+1}(A_\Gamma) = h^n \gamma_{i+1}(A_\Gamma)$.
\end{cor}
\begin{proof}
    We have $\nu(g \gamma_{i+1}(A_\Gamma)) = \mu(g)_i \in n \mathfrak{Lie}(\Z\langle\Gamma\rangle)$. By \cref{thm:LieAlg_multiples}, this implies that $g \gamma_{i+1}(A_\Gamma) \in n \mathfrak{gr}(A_\Gamma)$, which in turn implies there is some $h \in \gamma_i(A_\Gamma)$ such that $g = h^n$ modulo $\gamma_{i+1}(A_\Gamma)$. \qedhere
\end{proof}

Let $\Gamma$ be a finite simplicial graph and let $p$ be a prime. There are natural maps
\[
    \frac{\Z/p^{k+1}\langle\Gamma\rangle}{\omega_{\Z/p^{k+1}\langle\Gamma\rangle}^n} \to \frac{\Z/p^k\langle\Gamma\rangle}{\omega_{\Z/p^k\langle\Gamma\rangle}^n} \leftarrow \frac{\Z/p^k\langle\Gamma\rangle}{\omega_{\Z/p^k\langle\Gamma\rangle}^{n+1}}
\]
which form an inverse system as $n,k \in \N$ vary.

\begin{lem}\label{lem:top_structure}
    Let $\Gamma$ be a finite simplicial graph and let $p$ be a prime. There is an isomorphism
        \[
            \Z_p\llangle \Gamma \rrangle \cong \varprojlim_{n,k \in \N} \frac{\Z/p^k\langle\Gamma\rangle}{\omega_{\Z/p^k\langle\Gamma\rangle}^n}.
        \]
\end{lem}
\begin{proof}
    We have
    \[
        \varprojlim_{n,k} \frac{\Z/p^k\langle\Gamma\rangle}{\omega_{\Z/p^k\langle\Gamma\rangle}^n} = \varprojlim_n \varprojlim_k \frac{\Z/p^k\langle\Gamma\rangle}{\omega_{\Z/p^k\langle\Gamma\rangle}^n} = \varprojlim_n \frac{\Z_p\langle\Gamma\rangle}{\omega_{\Z_p\langle\Gamma\rangle}^n} = \Z_p\llangle\Gamma\rrangle. \qedhere
    \]
\end{proof}

\cref{lem:top_structure} gives $\Z_p\llangle\Gamma\rrangle$ the structure of a topological (in fact profinite) ring, induced by endowing each finite ring $\frac{\Z/p^k\langle\Gamma\rangle}{\omega^n}$ with the discrete topology. In the following lemma, the subset $1 + \omega_{\Z_p\llangle\Gamma\rrangle} \leqslant \Z_p\llangle\Gamma\rrangle$ is given the subspace topology.

\begin{lem}\label{lem:units_pro_p}
    If $\Gamma$ is a finite simplicial graph and let $p$ be a prime, then the group $1 + \omega_{\Z_p\llangle\Gamma\rrangle}$ is pro-$p$.
\end{lem}
\begin{proof}
    We claim that the ideal $\omega_{\Z_p\llangle\Gamma\rrangle}$ is closed in $\Z_p\llangle\Gamma\rrangle$. Let $\pi_{n,k}$ be the canonical projection 
    \[
        \Z_p\llangle\Gamma\rrangle \to \frac{\Z/p^k\langle\Gamma\rangle}{\omega_{\Z/p^k\langle\Gamma\rangle}^n}.
    \]
    By \cref{lem:top_structure},
    \[
        \omega_{\Z_p\llangle\Gamma\rrangle} \subseteq \overline{\omega_{\Z_p\llangle\Gamma\rrangle}} = \varprojlim_{n,k\in \N} \pi_{n,k}\left(\omega_{\Z_p\llangle\Gamma\rrangle} \right).
    \]
    If $f \notin \omega_{\Z_p\llangle\Gamma\rrangle}$, then $f_0 \neq 0$. But then $\pi_{n,k}(f)_0 \neq 0$ for any $n \geqslant 1$ and all sufficiently large $k$, and therefore
    \[
        f \notin \varprojlim_{n,k\in \N} \pi_{n,k}\left(\omega_{\Z_p\llangle\Gamma\rrangle} \right),
    \]
    which shows that $\omega_{\Z_p\llangle\Gamma\rrangle} = \overline{\omega_{\Z_p\llangle\Gamma\rrangle}}$.
    
    Since the map $x \mapsto x + 1$ is a homeomorphism, $1 + \omega_{\Z_p\llangle\Gamma\rrangle}$ is also closed. Hence, 
    \[
        1 + \omega_{\Z_p\llangle\Gamma\rrangle} = \varprojlim_{n,k \in \N} \pi_{n,k}\left( 1 + \omega_{\Z_p\llangle\Gamma\rrangle} \right).
    \]
    It is clear that $\pi_{n,k}\left( 1 + \omega_{\Z_p\llangle\Gamma\rrangle} \right)$ is a finite $p$-group, and therefore $1 + \omega_{\Z_p\llangle\Gamma\rrangle}$ is a pro-$p$ group. \qedhere
\end{proof}

We now show that the Magnus embedding induces a faithful representation of $\mathbf A_\Gamma$ into the power series ring $\Z_p\llangle\Gamma\rrangle$; this was done by Lazard in the case of a free group \cite{Lazard1954}. This is the pro-$p$ analogue of \cref{thm:DuchampKrob}.

\begin{thm}[Magnus embedding for pro-$p$ RAAGs]\label{thm:propMagnus}
    Let $\Gamma$ be a finite simplicial graph, let $A_\Gamma$ be the right-angled Artin group on $\Gamma$, let $p$ be a prime, and let $\mathbf A_\Gamma$ be the pro-$p$ completion of $A_\Gamma$. 
    \begin{enumerate}
        \item\label{item:cts_hom} There is a continuous homomorphism
        \[
            \mu_p \colon \mathbf A_\Gamma \to 1 + \omega_{\Z_p\llangle \Gamma \rrangle}
        \]
        determined by $\mu_p(s_v) = 1 + v$.
        \item\label{item:cts_hom_inj} The map $\mu_p$ is injective.
        \item\label{item:LCS_preimage} For all $n \in \N$, we have $\mu_p\inv(1 + \omega_{\Z_p\llangle\Gamma\rrangle}^n) = \gamma_n(\mathbf A_\Gamma)$.
    \end{enumerate}
\end{thm}

\begin{proof}
    By \cref{lem:units_pro_p}, $1 + \omega_{\Z_p\llangle \Gamma \rrangle}$ is a pro-$p$ group, so by the universal property of pro-$p$ completions there is a commutative square
    \[
        \begin{tikzcd}
            A_\Gamma \arrow[r,hook,"\mu"]\arrow[d,hook] & 1 + \omega_{\Z\llangle\Gamma\rrangle}\arrow[d,hook] \\
            \mathbf A_\Gamma \arrow[r,"\mu_p"] & 1 + \omega_{\Z_p\llangle\Gamma\rrangle} \nospacepunct{,}
        \end{tikzcd}
    \]
    where $\mu_p$ is continuous. This proves item \ref{item:cts_hom}.

    For any finitely generated group $G$, there is a natural isomorphism 
    \[
        \mathfrak{gr}(\mathbf{G}) \cong \Z_p \otimes \mathfrak{gr}(G).
    \]
    Moreover, there is an isomorphism
    \[
        \mathfrak{Lie}(\Z_p\langle\Gamma\rangle) \cong \Z_p \otimes \mathfrak{Lie}(\Z\langle\Gamma\rangle),
    \]
    so \cref{thm:LieAlg_multiples} yields the following commutative diagram of inclusions of Lie algebras
    \[
        \begin{tikzcd}
            \gr(A_\Gamma) \arrow[r,hook,"\nu"] \arrow[d,hook] & \mathfrak{Lie}(\Z\langle\Gamma\rangle) \arrow[d,hook] \\
            \gr(\mathbf A_\Gamma) \arrow[r,hook,"\Z_p \otimes \nu"] & \mathfrak{Lie}(\Z_p\langle\Gamma\rangle) \nospacepunct{.}
        \end{tikzcd}
    \]
    The injectivity of the lower map follows from the fact that $\Z_p$ is torsion-free. There is a second map
    \[
        \tau\colon \gr(\mathbf A_\Gamma) \to \mathfrak{Lie}(\Z_p\langle\Gamma\rangle), \quad g\gamma_{n+1}(\mathbf A_\Gamma) \mapsto \mu_p(g)_n,
    \]
    where we recall that $\mu_p(g)_n$ denotes the degree $n$ component of $\mu_p(g)$. This map is well defined, since it maps $\gamma_{n+1}(\mathbf A_\Gamma)$ to zero. Note that $\tau$ is continuous on each summand $\gamma_n(\mathbf A_\Gamma)/\gamma_{n+1}(\mathbf A_\Gamma)$ since $\mu_p$ is continuous. Moreover, $\tau$ agrees with $\Z_p \otimes \nu$ on the dense subset 
    \[
        \gamma_n(A_\Gamma)/\gamma_{n+1}(A_\Gamma) \subseteq \gamma_n(\mathbf A_\Gamma)/\gamma_{n+1}(\mathbf A_\Gamma),
    \]
    and therefore $\tau = \Z_p \otimes \nu$.

    We now show that $\mu_p$ is injective. Let $g \in \mathbf A_\Gamma$ be a non-trivial element. Since $\mathbf A_\Gamma$ is residually nilpotent, there is some $n \in \N$ such that $g \in \gamma_n(\mathbf A_\Gamma) \smallsetminus \gamma_{n+1}(\mathbf A_\Gamma)$. By injectivity of $\tau$, we have $\mu_p(g)_n = \tau(g\gamma_{n+1}(\mathbf A_\Gamma)) \neq 0$. This means that $\mu_p(g) \neq 0$, proving item \ref{item:cts_hom_inj}.

    We prove item \ref{item:LCS_preimage}. Since $\gamma_n(A_\Gamma)$ is mapped to $1 + \omega_{\Z_p}^n$ by \cref{thm:DuchampKrob} and $1 + \omega_{\Z_p}^n$ is closed, it follows that $\mu_p$ maps the closure $\gamma_n(\mathbf A_\Gamma)$ into $1 + \omega_{\Z_p}^n$ for all $n \in \N$. Conversely, suppose that $\mu_p(g) \in 1 + \omega_{\Z_p\llangle\Gamma\rrangle}^n$. If $g \notin \gamma_n(\mathbf A_\Gamma)$, then $g \in \gamma_c(\mathbf A_\Gamma) \smallsetminus \gamma_{c+1}(\mathbf A_\Gamma)$ for some $c < n$. Then 
    \[
        \tau(g\gamma_{c+1}(\mathbf A_\Gamma)) = \mu_p(g)_c = 0,
    \]
    contradicting the injectivity of $\tau$. So $g \in \gamma_n(\mathbf A_\Gamma)$, proving item \ref{item:LCS_preimage}. \qedhere
\end{proof}

We can use the Magnus embedding of the pro-$p$ RAAG $A_\Gamma$ to study the centre of $\gr(\mathbf A_\Gamma)$, much as in \cite[Section 3]{Wade_JohnsonHomsHigherRank}. We will need a result of Wade {\cite[Proposition 3.9]{Wade_JohnsonHomsHigherRank}}, which is stated for $\Z\langle\Gamma\rangle$. The coefficient ring plays no role and the same proof yields the following result over $\Z_p$.

\begin{prop}\label{prop:non-trivial-coefficient}
    Let $\Gamma$ be a finite simplicial graph such that $A_\Gamma$ is a centreless RAAG. If $x = \sum_{k=1}^n \lambda_k x_k$ is an element of $\Z_p\langle\Gamma\rangle$, where the elements $x_k$ are pairwise distinct monomials in $V(\Gamma)$, and $\lambda_k \in \Z_p$ for each $k$, then there is an element $v \in V(\Gamma)$ such that the coefficient of the term $g_k v$ in $[x,v]$ is $\lambda_k$.
\end{prop}

We can now apply \cref{thm:propMagnus} to obtain the pro-$p$ analogue of \cite[Theorem 3.11]{Wade_JohnsonHomsHigherRank}.

\begin{cor}\label{cor:trivial_centres}
    If $Z(A_\Gamma) = \{1\}$, then 
    \begin{enumerate}
        \item\label{item:centre_prop} $Z(\mathbf A_\Gamma) = \{1\}$;
        \item\label{item:centre_Lie} $Z(\gr(A_\Gamma)) = \{0\}$;
        \item\label{item:centre_Lie_prop} $Z(\gr(\mathbf A_\Gamma)) = \{0\}$;
        \item\label{item:centre_modp} $Z(\gr(A_\Gamma) \otimes \F_p) = Z(\gr(\mathbf A_\Gamma) \otimes_{\Z_p} \F_p) = \{0\}$.
    \end{enumerate}
\end{cor}

\begin{proof}
  Let $g \in \mathbf A_\Gamma$ be a non-trivial element and let $x$ be the non-trivial homogeneous component of $\mu_p(g)$ of least positive degree. By \cref{prop:non-trivial-coefficient}, there is some $v \in \Gamma$ such that $[x,v] \neq \{0\}$. But $[x,v]$ is the homogeneous component of $\mu_p([g,v])$ of least positive degree, and therefore $[g,v] \neq 1$. We conclude that $Z(\mathbf A_\Gamma) = \{1\}$. This proves item \ref{item:centre_prop}.
  
    Item \ref{item:centre_Lie} was proven by Wade in \cite[Theorem 3.11]{Wade_JohnsonHomsHigherRank}. For item \ref{item:centre_modp}, note that
    \[
        \gr(\mathbf A_\Gamma) \otimes_{\Z_p} \F_p \cong \gr(A_\Gamma) \otimes \Z_p \otimes_{\Z_p} \F_p \cong \gr(A_\Gamma) \otimes \F_p,
    \]
    where a tensor product with no subscript indicates the tensor is being taken over $\Z$. But $Z(\gr(A_\Gamma) \otimes \F_p) = \{0\}$ by \cite[Theorem 3.11]{Wade_JohnsonHomsHigherRank} as well.

    Let $y \in \gr(\mathbf A_\Gamma)$ be a non-trivial element. Write $y = y_{i_1} + \cdots + y_{i_n}$, where each $y_{i_j} \in \gamma_{i_j}(\mathbf A_\Gamma) /\gamma_{i_j+1}(\mathbf A_\Gamma)$ is a non-zero element and the indices $i_j$ are pairwise distinct. For each $v \in V(\Gamma)$, the decomposition
    \[
        [y,v] = [y_{i_1},v] + \cdots + [y_{i_n},v]
    \]
    is also into homogeneous components of pairwise different degrees, so it suffices to assume that $y$ was homogeneous to begin with. Write $y = g\gamma_{n+1}(\mathbf A_\Gamma)$ for some $n \in \N$. Since $\mu_p(g)_n$ is non-zero by \cref{thm:propMagnus}, it follows that $[\mu_p(g),v]_{n+1}$ is non-zero for some $v \in V(\Gamma)$ by \cref{prop:non-trivial-coefficient}. In particular, $\nu([y,v]) \neq 0$ so $y$ is non-central. This concludes the proof of item \ref{item:centre_Lie_prop}. \qedhere
\end{proof}

Combined with \cref{cor:trivial_centres}, Asada's \cref{thm:Asada} yields the following pro-$p$ analogue of \cite[Corollary 4.8]{Wade_JohnsonHomsHigherRank}.

\begin{cor}\label{cor:prop_RTFN_torelli}
    If $Z(\mathbf A_\Gamma)=\{1\}$, then the Andreadakis--Johnson filtration
    \[
        \mathcal T(\mathbf A_\Gamma) = \mathcal T^{(1)}(\mathbf A_\Gamma) \geqslant \mathcal T^{(2)}(\mathbf A_\Gamma) \geqslant \dots
    \]
    is residual and has the property that every quotient $\mathcal T^{(n)}(\mathbf A_\Gamma)/\mathcal T^{(n+1)}(\mathbf A_\Gamma)$ is a finitely generated free $\Z_p$-module. In particular, $\mathcal T(\mathbf A_\Gamma)$ is residually torsion-free nilpotent.
\end{cor}

If $\Lambda$ is a full subgraph of $\Gamma$, the Magnus embedding of \cref{thm:propMagnus} implies that $\mathbf A_\Lambda \leqslant \mathbf A_\Gamma$, since $\mathbf A_\Lambda$ can be thought of as sitting inside $\Z_p\llangle\Lambda\rrangle \subseteq \Z_p\llangle\Gamma\rrangle$ via the embedding $\mu_p$ of \cref{thm:propMagnus}. This observation also follows from \cite[Lemma 3.4]{Casals-RuizPintonelloZalesskii2025}. The centre of an arbitrary RAAG $A_\Gamma$ is isomorphic to $A_\Lambda$, where $\Lambda$ is a maximal clique subgraph of $\Gamma$ such that $\Gamma$ decomposes as a join with $\Lambda$. Algebraically, this means that $A_\Gamma$ decomposes as a direct product with $A_\Lambda$.

\begin{lem}
    If $\Lambda \leqslant \Gamma$ is the full clique subgraph such that $Z(A_\Gamma) = A_\Lambda$, then $Z(\mathbf A_\Gamma) = \mathbf A_\Lambda$.
\end{lem}
\begin{proof}
    It is clear that $\mathbf A_\Lambda \leqslant Z(\mathbf A_\Gamma)$ since the centre is a closed subgroup containing $A_\Lambda$, so we focus on the reverse inclusion. Let $\Upsilon = \Gamma \smallsetminus \Lambda$ be the full subgraph of $\Gamma$ on the vertices $V(\Gamma) \smallsetminus V(\Lambda)$. Then $A_\Gamma/Z(A_\Lambda) \cong A_{\Upsilon}$.

    There is a short exact sequence $1 \to A_\Lambda \to A_\Gamma \to A_{\Upsilon} \to 1$, where $A_{\Upsilon}$ is a centreless RAAG. Taking pro-$p$ completions, we have
    \[
        1 \to \mathbf A_\Lambda \to \mathbf A_\Gamma \to \mathbf A_\Upsilon \to 1,
    \]
    where we have used the fact that the closure of $A_\Lambda$ in $\mathbf A_\Gamma$ is isomorphic to $\mathbf A_\Lambda$ (see the remarks preceding the lemma). By \cref{cor:trivial_centres}, $Z(\mathbf A_\Upsilon) = \{1\}$, which implies that $Z(\mathbf A_\Gamma) = \mathbf A_\Lambda$. \qedhere
\end{proof}

We continue with the notation of the previous lemma: let $\Gamma$ be a finite simplicial graph, let $\Lambda \leqslant \Gamma$ be the full clique subgraph such that $Z(\mathbf A_\Gamma) = \mathbf A_\Lambda$, and let $\Upsilon = \Gamma \smallsetminus \Lambda$. Every automorphism $\varphi \in \Aut(\mathbf A_\Gamma)$ induces an automorphism of $\mathbf A_\Lambda$ and therefore of $\mathbf A_\Upsilon = \mathbf A_\Gamma/\mathbf A_\Lambda$, since the centre is characteristic. There is thus a well defined homomorphism
\[
    \pi \colon \Out(\mathbf A_\Gamma) \to \Out(\mathbf A_\Upsilon),
\]
since an inner automorphism of $\mathbf A_\Gamma$ induces an inner automorphism of $\mathbf A_\Upsilon$.

We now prove a pro-$p$ analogue of a trick of Charney and Vogtmann \cite[Proposition 4.4]{CharneyVogtmann2009} which allows us to upgrade \cref{cor:prop_RTFN_torelli} to all pro-$p$ RAAGs.

\begin{thm}\label{thm:prop_RTFN_Torelli}
    The homomorphism $\pi$ restricts to an isomorphism
    \[
        \pi|_{\mathcal T^{(n)}(\mathbf A_\Gamma)} \colon \mathcal T^{(n)}(\mathbf A_\Gamma) \to \mathcal T^{(n)}(\mathbf A_\Upsilon)
    \]
    for each $n \in \N$. In particular, the Andreadakis--Johnson filtration
    \[
        \mathcal T(\mathbf A_\Gamma) = \mathcal T^{(1)}(\mathbf A_\Gamma) \geqslant \mathcal T^{(2)}(\mathbf A_\Gamma) \geqslant \dots
    \]
    is residual and every quotient $\mathcal T^{(n)}(\mathbf A_\Gamma)/\mathcal T^{(n+1)}(\mathbf A_\Gamma)$ is a finitely generated free $\Z_p$-module.
\end{thm}
\begin{proof}
    It is not difficult to verify that $\pi$ maps $\mathcal T^{(n)}(\mathbf A_\Gamma)$ into $\mathcal T^{(n)}(\mathbf A_\Upsilon)$ for each $n \in \N$, so the restricted maps in the statement of the theorem are well defined. We begin by showing that $\pi$ is injective on $\mathcal T(\mathbf A_\Gamma)$. Let $\varphi \in \Aut(\mathbf A_\Gamma)$ be an automorphism such that $[\varphi] \in \ker(\pi) \cap \mathcal T(\mathbf A_\Gamma)$. Since $\varphi$ is inner on $\mathbf A_\Upsilon$, there is an element $g \in \mathbf A_\Gamma$ such that for every $v \in V(\Gamma)$ we have $\varphi(v) = g\inv s_v g z_v$ for some $z_v \in \mathbf A_\Lambda$. Hence, on the Abelianisation $\mathbf A_\Gamma/\gamma_2(\mathbf A_\Gamma)$ we have that $\varphi$ sends $\overline{s_v}$ to $\overline{s_v} + \overline{z_v}$, where $\overline{s_v}$ and $\overline{z_v}$ denote the images of $s_v$ and $z_v$ in $\mathbf A_\Gamma/\gamma_2(\mathbf A_\Gamma)$. Since we have assumed that $\varphi$ is trivial on the Abelianisation, we conclude that $\overline{z_v} = 0$. But the Abelianisation of $\mathbf A_\Gamma$ is just $\Z_p^{|V(\Gamma)|}$, so $\mathbf A_\Lambda$ injects into $\mathbf A_\Gamma/\gamma_2(\mathbf A_\Gamma)$, which implies that $z_v = 1$ for each $v$. Hence, $\varphi$ is the inner automorphism given by conjugation by $g$, which shows that $\pi$ is injective on $\mathcal T(\mathbf A_\Gamma)$ (and hence that each restriction to $\mathcal T^{(n)}(\mathbf A_\Gamma)$ is injective).

    Next we check surjectivity. Fix $n \in \N$ and let $\varphi \in \Aut(\mathbf A_\Upsilon)$ be such that $[\varphi] \in \mathcal T^{(n)}(\mathbf A_\Upsilon)$. We can define a lift $\widehat{\varphi} \colon \mathbf A_\Gamma \to \mathbf A_\Gamma$ of $\varphi$ by setting 
    \[
        \widehat{\varphi}(s_v) = \begin{cases}
            \varphi(s_v) & \text{if} \ v \in \Upsilon \\
            s_v & \text{if} \ v \in \Lambda,
        \end{cases}
    \]
    which is well defined because the generators $s_v$ for $v \in \Lambda$ are central. We check that $\widehat{\varphi}$ is indeed an automorphism. Suppose that $g \in \mathbf A_\Gamma$ is such that $\widehat{\varphi}(g) = 1$. The commutative square
    \[
        \begin{tikzcd}
            \mathbf A_\Gamma \arrow[r, "\widehat{\varphi}"] \arrow[d, two heads] & \mathbf A_\Gamma \arrow[d, two heads] \\
            \mathbf A_\Upsilon \arrow[r, "\varphi"] & \mathbf A_\Upsilon
        \end{tikzcd}
    \]
    shows that $\varphi(\overline{g}) = 1$, where $\overline{g}$ is the image of $g$ in $\mathbf A_\Upsilon$. Since $\varphi$ is an isomorphism, we must have $\overline g = 1$, so $g \in Z(\mathbf A_\Gamma) = \mathbf A_\Lambda$. But $\widehat{\varphi}$ fixes $\mathbf A_\Lambda$, so $g = 1$, proving that $\widehat{\varphi}$ is injective. Now suppose that $g \in \mathbf A_\Gamma$ is arbitrary. From the commutative square, we see there is some $h \in \mathbf A_\Gamma$ such that $\widehat\varphi(h)g\inv = z \in \mathbf A_\Lambda$. But then we have
    \[
        \widehat\varphi(z\inv h) = \widehat\varphi(z\inv)\widehat\varphi(h) = z\inv \cdot z g = g,
    \]
    which proves surjectivity. Hence, we have constructed the desired outer class $[\widehat\varphi] \in \mathcal T^{(n)}(\mathbf A_\Gamma)$ which maps to $[\varphi]$ under $\pi$. This conclude the proof that the restrictions of $\pi$ are all isomorphisms.
    
    The rest of the theorem now follows immediately from \cref{cor:prop_RTFN_torelli}. \qedhere
\end{proof}

\section{The Atiyah and Zero Divisor Conjectures}\label{sec:main_results}

\subsection{Automorphism groups}\label{subsec:aut}

We begin by showing that automorphism groups of certain residually torsion-free nilpotent groups satisfy the Weak Atiyah Conjecture over $\C$ and for each prime $p$ admit a finite-index subgroup whose group algebras over fields of characteristic $p$ or zero do not have any zero divisors. We say that $G$ is a \emph{Magnus group} if $G$ is residually nilpotent and $G/\gamma_n(G)$ is torsion-free for all $n \in \N$. Important examples of Magnus groups include free groups, parafree groups, fundamental groups of orientable surfaces, and right-angled Artin groups. Moreover, the pro-$p$ completion of a finitely generated Magnus group is again a Magnus group.

\begin{thm}\label{thm:aut_atiyah}
    If $\mathbf G$ is a finitely generated pro-$p$ Magnus group for some prime $p$, then
    \begin{enumerate}
        \item\label{item:aut_atiyah_prop} $T_p(\mathbf G)$ satisfies the Strong Atiyah Conjecture over $\C$;
        \item\label{item:aut_weak_atiyah} $\Aut(\mathbf G)$ satisfies the Weak Atiyah Conjecture over $\C$;
        \item\label{item:aut_div_ring} the completed group algebra $\K\llbracket  T_p(\mathbf G)\rrbracket$ embeds into a division ring for any field $\K$ of characteristic $p$.
    \end{enumerate}
\end{thm}

\begin{rem}
    Since $\K[T_p(\mathbf G)]$ embeds into $\K\llbracket T_p(\mathbf G)\rrbracket$, item \ref{item:aut_div_ring} implies that the Zero Divisor Conjecture holds for $\K[T_p(\mathbf G)]$, where $\K$ is a field of characteristic $p$.
\end{rem}

To prove item \ref{item:aut_div_ring}, we need an auxiliary lemma. It is likely well known, but we could not find a reference.

\begin{lem}\label{lem:contained_in_open}
    Let $G$ be a compact topological group, $U$ be an open subgroup of $G$, and, for $n \geqslant 1$, $H_n$ a nested sequence of closed subgroups such that $\bigcap_{n=1}^{\infty} H_n= \{1\}$. For sufficiently large $n$, we have $H_n \subset U$.
\end{lem}
\begin{proof}
    If the conclusion of the lemma is false, then there is a sequence of elements $(g_n)_{n \in \N}$ such that $g_n \in H_n \smallsetminus U$ for all $n \in \N$. Since $G$ is compact, $(g_n)$ has a convergent subsequence whose limit we denote by $g$. The nested subgroups $H_n$ intersect in the identity, forcing $g = 1$. On the other hand, every element $g_n$ lies in the closed set $G \smallsetminus U$ and therefore so does $g$. This forces $g \neq 1$, a contradiction.\qedhere
\end{proof}

\begin{proof}[Proof (of \cref{thm:aut_atiyah})]
    By work of Andreadakis, the terms 
    \[
        T(\mathbf G) = T^{(1)}(\mathbf G) \geqslant T^{(2)}(\mathbf G) \geqslant \dots
    \]
    of the Andreadakis--Johnson filtration of $T(\mathbf G)$ form a central series intersecting in the identity  \cite[Theorems 1.1 and 1.2] {Andreadakis1965}. The following claim is the natural pro-$p$ analogue of \cite[Corollary 2]{Asada1995}, and its proof is similar.

    \begin{claim}\label{claim:andreadakis_johnson_torsionfree}
        The consecutive quotient $T^{(n)}(\mathbf G)/T^{(n+1)}(\mathbf G)$ is a finitely generated free $\Z_p$-module for each $n \in \N$.
    \end{claim}
    \begin{proof}
        Let $\{x_1, \dots, x_m\}$ be a finite (topological) generating set for $\mathbf G$. For each $\varphi \in \Aut(\mathbf G)$, let $s_i(\varphi) := \varphi(x_i) x_i\inv$. For each $n \in \N$, we claim that the map
        \[
            \Psi_n \colon T^{(n)}(\mathbf G)/T^{(n+1)}(\mathbf G) \to \bigoplus_{i = 1}^m \gamma_{n+1}(\mathbf G)/\gamma_{n+2}(\mathbf G)
        \]
        defined by $\Psi_n(\varphi T^{(n+1)}(\mathbf G)) = (s_i(\varphi)\gamma_{n+2}(\mathbf G))_{1 \leqslant i \leqslant m}$ is a well defined injective homomorphism.

        Indeed, if $\varphi \in T^{(n+1)}(\mathbf G)$, then $\varphi(x_i) = x_i c_i$ for some term $c_i \in \gamma_{n+2}(\mathbf G)$, by definition. This implies that $\Psi_n$ maps $T^{(n+1)}$ to $\bigoplus_{i=1}^m \gamma_{m+2}(\mathbf G)$ (for all $n \in \N)$, and therefore that $\Psi_n$ is well defined. Next, if $\varphi \in T^{(n)}(\mathbf G) \smallsetminus T^{(n+1)}(\mathbf G)$, then there is some generator $x_i$ for which $\varphi(x_i) \neq x_i$ modulo $\gamma_{n+2}(\mathbf G)$, or equivalently such that $s_i(\varphi) \notin \gamma_{n+2}(\mathbf G)$. This shows that $\Psi_n$ is injective. Finally, we check that $\Psi_n$ is a homomorphism. By the identity
        \begin{align*}
            s_i(\varphi \psi) &= \varphi(\psi(x_i))x_i\inv = \varphi(\psi(x_i)x_i\inv x_i)x_i\inv \\
            &= \varphi(\psi(x_i)x_i\inv) \varphi(x_i)x_i\inv = \varphi(s_i(\psi)) s_i(\varphi),
        \end{align*}
        it suffices to show that $\varphi \in T^{(n)}(\mathbf G)$ acts trivially on $\gamma_{n+1}(\mathbf G)/\gamma_{n+2}(\mathbf G)$. Let $[x,y] \in \gamma_{n+1}(\mathbf G)$, where $x \in \mathbf G$ and $y \in \gamma_n(\mathbf G)$. Then 
        \[
            \varphi([x,y]) = [\varphi(x), \varphi(y)] = [xa, yb]
        \]
        for some $a,b \in \gamma_{n+1}(\mathbf G)$. But
        \[
            [xa, yb] = [x,yb]^a [a,yb] = [x,b]^a [x,y]^{ab} [a,b] [a,y]^b,
        \]
        which equals $[x,y]^{ab}$ reduced modulo $\gamma_{n+2}(\mathbf G)$. But again, 
        \[
            [x,y]^{ab} = [x,y] [[x,y],ab]
        \]
        equals $[x,y]$ modulo $\gamma_{n+2}(\mathbf G)$, proving that $\varphi$ acts trivially on $\gamma_{n+1}(\mathbf G)/\gamma_{n+2}(\mathbf G)$. This concludes the proof that $\Psi_n$ is an injective homomorphism.

        Since $\mathbf G$ is a (topologically) finitely generated Magnus group, $\gamma_{n+1}(\mathbf G)/\gamma_{n+2}(\mathbf G)$ is a finitely generated free $\Z_p$-module. Hence, $T^{(n)}(\mathbf G)/T^{(n+1)}(\mathbf G)$ is also a finitely generated free $\Z_p$-module. \renewcommand\qedsymbol{$\diamond$}\qedhere
    \end{proof}

    Since $\mathbf G$ is a finitely generated Magnus group, there is some $m \in \N$ such that $\mathbf G/\gamma_2(\mathbf G) \cong \Z_p^m$. If $\mathbf Q$ is the image of $T_p(\mathbf G)$ in $\GL_m(\Z_p)$, then for every $n \in \N$ there is a short exact sequence
    \[
        1 \to T(\mathbf G)/T^{(n)}(\mathbf G) \to T_p(\mathbf G)/T^{(n)}(\mathbf G) \to \mathbf Q \to 1.
    \]
    Since $\mathbf Q$ is a closed torsion-free subgroup of $\GL_m(\Z_p)$, it is torsion-free $p$-adic analytic. By \cref{claim:andreadakis_johnson_torsionfree}, $T(\mathbf G)/T^{(n)}(\mathbf G)$ is finitely generated torsion-free nilpotent for every $n \in \N$, and is therefore $p$-adic analytic as well. By \cref{lem:p_adic_extension}, $T_p(\mathbf G)/T^{(n)}(\mathbf G)$ is torsion-free $p$-adic analytic, and therefore satisfies the Strong Atiyah Conjecture over $\C$ by \cref{thm:p-adic-atiyah}. By \cref{cor:approx_Atiyah}, $T_p(\mathbf G)$ satisfies the Strong Atiyah Conjecture over $\C$, proving item \ref{item:aut_atiyah_prop}. Item \ref{item:aut_weak_atiyah} then follows immediately from item \ref{item:aut_atiyah_prop} and  \cref{lem:weak_atiyah_commensurability}.

    By \cref{thm:p_adic_completed_domain}, each completed group algebra $\K\llbracket T_p(\mathbf G)/T^{(n)}(\mathbf G) \rrbracket$ embeds into a division ring $\mathcal D_n$. By \cref{lem:contained_in_open}, for any open subgroup $\mathbf{U}$ there is some $n \in \N$ be such that $T^{(n)}(\mathbf G) \leqslant \mathbf U$. The projection $\K\llbracket T_p(\mathbf G)\rrbracket \to \K[T_p(\mathbf G)/\mathbf U]$ factors through
        \[
            \K\llbracket T_p(\mathbf G)\rrbracket \to \K\llbracket T_p(\mathbf G)/T^{(n)}(\mathbf G)\rrbracket \to \K[T_p(\mathbf G)/\mathbf U].
        \]
    Hence, $x \neq 0$ in $\K\llbracket T_p(\mathbf G)/T^{(n)}(\mathbf G)\rrbracket$. 
    
    This implies that $\K\llbracket T_p(\mathbf G) \rrbracket$ embeds into the ultraproduct $\prod_\omega \mathcal D_n$ for any non-principal ultrafilter $\omega$ on $\N$, proving item \ref{item:aut_div_ring}. \qedhere
\end{proof}

We can use \cref{thm:aut_atiyah} to obtain corresponding results for automorphism groups of finitely generated abstract Magnus groups.

\begin{cor}\label{cor:aut_atiyah}
    If $G$ is a finitely generated Magnus group, then
    \begin{enumerate}
        \item $T_p(G)$ satisfies the Strong Atiyah Conjecture over $\C$;
        \item $\Aut(G)$ satisfies the Weak Atiyah Conjecture over $\C$;
        \item the group algebra $\K[T_p(G)]$ embeds into a division ring for any field $\K$ of characteristic zero or $p$.
    \end{enumerate}
\end{cor}
\begin{proof}
    Let $\mathbf G$ be the pro-$p$ completion of $G$. If $g \in \mathbf G$ is an arbitrary non-identity element, then there is a finite $p$-quotient $Q$ of $\mathbf G$ in which $g$ survives. Then $Q$ is also a $p$-quotient of $G$, and since $Q$ is nilpotent, there is some $n \in \N$ such that $G \to Q$ factors as $G \to G/\gamma_n(G) \to Q$. But then $g$ survives in the quotient $\mathbf G/\gamma_n(\mathbf G)$ of $\mathbf G$, and $\mathbf G/\gamma_n(\mathbf G)$ is torsion-free nilpotent, being the pro-$p$ completion of a torsion-free nilpotent group. The subgroups $\gamma_n(\mathbf G)$ therefore intersect in the identity and each quotient $\mathbf G/\gamma_n(\mathbf G)$ is torsion-free nilpotent. Hence, $\mathbf G$ is a finitely generated pro-$p$ Magnus group. It is easy to see that $\Aut(G)$ embeds into $\Aut(\mathbf G)$ and $T_p(G)$ embeds into $T_p(\mathbf G)$, so the corollary now follows at once from \cref{lem:SAC_subgroups,lem:sofic_SAC_all_fields} and \cref{thm:aut_atiyah}. \qedhere
\end{proof}

\subsection{Outer automorphism groups}\label{subsec:out}

For outer automorphism groups, the situation is more delicate. For instance, if $\mathbf G$ is the pro-$p$ completion of an abstract group $G$, it is not so clear when the natural map $\Out(G) \to \Out(\mathbf G)$ is injective, even in the case where $G$ is residually $p$. Segal gave examples of finitely generated torsion-free nilpotent groups for which this map is not injective \cite{Segal90}. Moreover, it is more difficult to prove that the Torelli subgroup of $\Out(G)$ and $\Out(\mathbf G)$ are residually torsion-free nilpotent when $G$ is a finitely Magnus group. Fortunately, for many groups of interest, such as surface groups, free groups, and more generally right-angled Artin groups, all of these properties hold and we can prove results analogous to \cref{thm:aut_atiyah} and \cref{cor:aut_atiyah}. The proofs are completely analogous to those of \cref{subsec:aut}, so we will mostly only sketch them.

\begin{thm}\label{thm:out_Atiyah_prop}
    If $\mathbf G$ is the pro-$p$ completion of a surface group or of a finitely generated RAAG, then 
    \begin{enumerate}
        \item\label{item:out_prop_SAC} $\mathcal T_p(\mathbf G)$ satisfies the Strong Atiyah Conjecture over $\C$;
        \item\label{item:out_prop_WAC} $\Out(\mathbf G)$ satisfies the Weak Atiyah Conjecture over $\C$;
        \item\label{item:out_prop_div} the completed group algebra $\K\llbracket \mathbf {\mathcal T_p(\mathbf G)}\rrbracket$ embeds into a division ring for any field $\K$ of characteristic $p$. 
    \end{enumerate}
\end{thm}
\begin{proof}
    The proof is similar to that of \cref{thm:aut_atiyah}. The chain
    \[
        \mathcal T(\mathbf G) = \mathcal T^{(1)}(\mathbf G) \geqslant \mathcal T^{(2)}(\mathbf G) \geqslant \dots
    \]
    is residual and the consecutive quotients $\mathcal T^{(n)}(\mathbf G)/\mathcal T^{(n+1)}(\mathbf G)$ are finitely generated free $\Z_p$-modules for all $n \in \N$, by \cite[Theorem 1 (pro-$\ell$ case) and Section 3.2] {Asada1995} in the case $\mathbf G$ is a pro-$p$ surface group and by \cref{thm:prop_RTFN_Torelli} in the case $\mathbf G$ is a pro-$p$ RAAG. The proof now proceeds exactly as in that of \cref{thm:aut_atiyah}. \qedhere
\end{proof}

We can now apply \cref{thm:out_Atiyah_prop} to outer automorphism groups of certain discrete groups.

\begin{cor}\label{cor:out_atiyah}
    If $G$ is either a surface group or a finitely generated right-angled Artin group, then
    \begin{enumerate}
        \item\label{item:Out_SAC} $\mathcal T_p(G)$ satisfies the Strong Atiyah Conjecture over $\C$;
        \item\label{item:Out_WAC} $\Out(G)$ satisfies the Weak Atiyah Conjecture over $\C$;
        \item\label{item:Out_div_ring} the group algebra $\K[\mathcal T_p(G)]$ embeds into a division ring for any field $\K$ of characteristic zero or $p$.
    \end{enumerate}
\end{cor}
\begin{proof}
    First consider the case where $G$ is a finitely generated right-angled Artin group. We claim that the natural map
    \[
        \Out(G) \to \Out(\mathbf G)
    \]
    is injective. Suppose that $\varphi \in \Aut(G)$ is a non-inner automorphism; we will show that $\varphi$ induces a non-inner automorphism of $\mathbf G$. By a result of Minasyan \cite[Proposition 6.9]{Minasyan2012}, $\varphi$ is not pointwise inner, meaning there is some $g \in G$ such that $g$ and $\varphi(g)$ are non-conjugate. RAAGs are conjugacy $p$-separable by a theorem of Toinet \cite[Theorem 6.15]{Toinet2013}, so there is a characteristic $p$-quotient of $G$ in which the images of $g$ and $\varphi(g)$ are non-conjugate. But then $\varphi$ induces a non-inner automorphism of $\mathbf G$. In particular, this shows that $\mathcal T_p(G)$ injects into $\mathcal T_p(\mathbf G)$. Since both groups are torsion-free, \cref{lem:SAC_subgroups} and \cref{thm:out_Atiyah_prop} imply that $\mathcal T_p(G)$ also satisfies the Strong Atiyah Conjecture over $\C$. Since $\Out(G)$ is a finite extension of $\mathcal T_p(G)$, item \ref{item:Out_WAC} follows from \cref{lem:weak_atiyah_commensurability}. The characteristic zero case of item \ref{item:Out_div_ring} follows from item \ref{item:Out_SAC} and \cref{lem:sofic_SAC_all_fields}. The characteristic $p$ case follows directly from \cref{thm:out_Atiyah_prop} and the fact that $\K[\mathcal T_p(G)]$ embeds into $\K\llbracket \mathcal T_p(\mathbf G)\rrbracket$.
    
    Now consider the case $G = \pi_1(\Sigma)$ for a closed surface $\Sigma$ of genus $g$. As above, it suffices to show that $\Out(G) \to \Out(\mathbf G)$ is injective. But, as in the previous paragraph, this follows from the fact that surface groups are conjugacy $p$-separable \cite[Theorem 1.7]{Paris2009} and all of their pointwise inner automorphisms are inner \cite[Proof of Theorem 3]{Grossman1974}. The rest of the proof concludes as in the previous paragraph. \qedhere
\end{proof}

We note that although the Strong Atiyah Conjecture is not known to be closed under direct products, our method of proof shows that any direct product of groups of the form $\mathcal T_p(\mathbf G)$ satisfies the Strong Atiyah Conjecture over $\C$, and among torsion-free groups this passes to subgroups. There are many interesting subgroups of $\mathcal T_p(\mathbf G)$ as $G$ varies.

\begin{cor} \label{cor:congruence_atiyah}
    Let $\Sigma$ be a compact surface (possibly with boundary and finitely many marked points) and let $V$ be a handlebody. Then $\Mod(\Sigma)$ and $\Mod(V)$ satisfy the Weak Atiyah Conjecture over $\C$. Moreover, for every field $\K$, there are torsion-free subgroups $\Gamma \leqslant \Mod(\Sigma)$ and $\Lambda \leqslant \Mod(V)$ of finite index such that $\K[\Gamma]$ and $\K[\Lambda]$ embed into division rings.
\end{cor}
\begin{proof}
    If $\Sigma$ is a closed surface, then $\Mod(\Sigma)$ embeds into $\Out(\pi_1(\Sigma))$ as a subgroup of index two. By \cref{cor:out_atiyah}, $\Out(\pi_1(\Sigma))$ has a torsion-free finite-index subgroup satisfying the Strong Atiyah Conjecture over $\C$, hence so does $\Mod(\Sigma)$. Hence, $\Mod(\Sigma)$ satisfies the Weak Atiyah Conjecture over $\C$. Similarly, $\Out(\pi_1(\Sigma))$ has a finite-index subgroup whose group algebra over $\K$ embeds into a division ring, and this property passes to subgroups.

    The natural map $\Mod(V) \to \Mod(\partial V)$ induced by restricting homeomorphisms of $V$ to its boundary $\partial V$ is injective \cite[Section 3]{Hensel_Survey}. Since $\partial V$ is a closed surface, one then draws the same conclusions for the handlebody group $\Mod(V)$ as above.

    Finally, suppose $\Sigma$ has a non-empty boundary or that its set of marked points $\{p_1, \dots, p_n\}$ is non-empty, and let $X = \Sigma \smallsetminus \{p_1, \dots, p_n\}$. Then the natural map
    \[
        \Mod(\Sigma) \to \Out(\pi_1(X))
    \]
    is injective. To see this, note that since $X$ is aspherical, there is a one-to-one correspondence between homotopy classes of maps $X \to X$ and conjugacy classes of homomorphisms $\pi_1(X) \to \pi_1(X)$. Hence, the image of an element $[\varphi]$ of $\Mod(\Sigma)$ is trivial in $\Out(\pi_1(X))$ if and only if $\varphi$ is homotopic to the identity map. But then $\varphi$ is also isotopic to the identity by \cite[Theorem 6.4]{Epstein66}, so $[\varphi]=1$. 
    
    Now, $\pi_1(\Sigma \smallsetminus \{p_1, \dots, p_n\})$ is a free group, so its outer automorphism group has a torsion-free finite-index subgroup satisfying the Strong Atiyah Conjecture over $\C$ and such that its group algebra over $\K$ embeds into a division ring by \cref{cor:out_atiyah}. These properties pass to $\Mod(\Sigma)$. \qedhere
\end{proof}

\appendix
\section{\texorpdfstring{$p$}{p}-inner automorphisms of free nilpotent groups}\label{sec:appendix}
\smallskip
\begin{center} \textsc{by Sam P.\ Fisher}\end{center}
\medskip

In \cite{Segal90}, Segal discusses two generalisations of inner automorphisms for a group $G$: pointwise inner and $p$-inner automorphisms. We say that $\varphi \in \Aut(G)$ is \emph{pointwise inner} if $\varphi(g)$ is conjugate to $g$ for every $g \in G$, and $\varphi$ is \emph{$p$-inner} for a prime $p$ if $\varphi$ induces an inner automorphism of every finite $p$-quotient of $G$. Note that an automorphism is $p$-inner if and only if the natural map
\[
    \Out(G) \to \Out(\mathbf G)
\]
is injective (where $\mathbf G$ is the pro-$p$ completion of $G$); this was a key property in the proof of the Weak Atiyah Conjecture for $\Out(A_\Gamma)$ in the main article.

Segal \cite{Segal90} provides examples of finitely generated torsion-free nilpotent groups admitting non-inner pointwise inner as well as non-inner $p$-inner automorphisms. As a counterpoint, he also shows that for every finitely generated nilpotent group $G$, there is a finite list of primes $\{p_1, \dots, p_n\}$ such that if $\varphi \in \Aut(G)$ is $p_i$-inner for each $i$, then $\varphi$ is inner.

In \cite{Endimioni2002}, Endimioni shows that pointwise inner automorphisms of free nilpotent groups are necessarily inner. The goal of this appendix is to show that $p$-inner automorphisms of finitely generated free nilpotent groups are also inner.

\begin{thm}\label{thm:p_inner_free_nilpotent}
    Let $N$ be a finitely generated free nilpotent group, let $p$ be a prime, and let $\mathbf N$ denote the pro-$p$ completion of $N$. The natural map
    \[
        \Out(N) \to \Out(\mathbf N)
    \]
    is injective.
\end{thm}

A sufficient condition for $p$-inner automorphisms of $G$ to be inner is that $G$ is that $G$ be conjugacy $p$-separable and that all its pointwise inner automorphisms are inner. While Endimioni's result provides one half of this criterion, finitely generated torsion-free nilpotent groups are only conjugacy $p$-separable if they are Abelian \cite{Ivanova2002}, so more work is necessary to establish \cref{thm:p_inner_free_nilpotent}.

\subsection{Setup and notation}

We will prove \cref{thm:p_inner_free_nilpotent} for a class of nilpotent groups properly containing the class of free nilpotent groups. Let $\Gamma$ be a simplicial graph and let $A_\Gamma$ be the associated RAAG. The \emph{free partially commutative nilpotent group of class $c$} is
\[
    N_{c,\Gamma} = A_\Gamma/\gamma_{c+1}(A_\Gamma).
\]
These are the free $\Gamma$-partially commutative objects in the variety of class at most $c$ nilpotent groups. If $\Gamma$ has no edges, then $N_{c,\Gamma}$ is the free nilpotent group of class $c$ generated by $V(\Gamma)$.

Throughout the rest of the appendix, the graph $\Gamma$ will be fixed. For each vertex $v$ of $\Gamma$, we denote by $s_v$ the corresponding generator of $\mathbf A_\Gamma$. We also fix a prime $p$ and denote by $\mathbf A_\Gamma$ (resp.\ $\mathbf N_{c,\Gamma}$) the pro-$p$ completion of $A_\Gamma$ (resp.\ $N_{c,\Gamma}$). By \cref{thm:DuchampKrob}, there is an embedding
\[
    \mu \colon A_\Gamma \to 1 + \omega_{\Z\llangle \Gamma \rrangle}, \quad s_v \mapsto 1 + v
\]
such that $\mu\inv(1 + \omega_{\Z\llangle \Gamma \rrangle}^i) = \gamma_i(A_\Gamma)$ for all $i \in \N$. It now follows that $N_{c,\Gamma}$ is torsion-free nilpotent and therefore is residually $p$.

By \cref{thm:propMagnus} $\mu$ induces an embedding
\[
    \bm{\mu} \colon \mathbf A_\Gamma \to 1 + \omega_{\Z_p\llangle \Gamma \rrangle}
\]
of the pro-$p$ completion with the corresponding property $\bm{\mu}\inv(1 + \omega_{\Z_p\llangle \Gamma \rrangle}^i) = \gamma_i(\mathbf A_\Gamma)$ for all $i \in \N$ for all $i \in \N$.

Recall that there is a map
\[
    \nu \colon \gr(A_\Gamma) \hookrightarrow \mathfrak{Lie}(\Z\langle \Gamma \rangle)
\]
defined by sending $g\gamma_{i+1}(A_\Gamma)$ to the degree $i$ component of $\mu(g)$ in $\Z\llangle \Gamma \rrangle$, where $g \in \gamma_{i}(A_\Gamma)$. This extends to a map
\[
    \bm\nu = \Z_p \otimes \nu \colon \gr(\mathbf A_\Gamma) \to \mathfrak{Lie}(\Z_p\langle \Gamma\rangle),
\]
which can also be defined by sending $g\gamma_{i+1}(\mathbf F)$ to the degree $i$ component of $\bm\mu(g)$ in $\Z_p\llangle X\rrangle$ for $g \in \gamma_i(\mathbf F)$. Both $\nu$ and $\bm\nu$ are injective Lie algebra homomorphisms.

We will need the notion of a Lyndon heap, which was not introduced in the main article. Lyndon heaps are special kinds of bases for $\gr(A_\Gamma)$; they were originally defined by Lalonde in \cite{Lalonde_basesLyndon93}. It will suffice to record some of their properties (as listed in \cite[Theorem 3.6]{Wade_JohnsonHomsHigherRank}) rather than provide a complete definition.

\begin{thm}[{\cite{Lalonde_basesLyndon93}}]\label{thm:Lyndon_basis}
    Let $M$ be the monoid of positive words in the letters $V(\Gamma)$. There is a subset $L \subseteq M$, a total order $\prec$ on $M$, and a set 
    \[
        \mathcal B = \{\beta_w : w \in \mathcal L\} \subseteq \gr(A_\Gamma)
    \]
    indexed by $\mathcal L$ such that
    \begin{enumerate}
        \item $\mathcal B$ is a $\Z$-basis for $\gr(A_\Gamma)$, viewed as a free $\Z$-module;
        \item the coefficient of $w$ in $\nu(\beta_w)$ is $1$;
        \item if the coefficient of $u \in M$ in $\nu(\beta_w)$ is non-zero, then $w \preccurlyeq u$.
    \end{enumerate}
\end{thm}

The theorem immediately implies that $\mathcal B$ is also a $\Z_p$-basis for $\gr(\mathbf A_\Gamma)$, that the coefficient of $w$ in $(\Z_p \otimes \nu)(\beta_w)$ is $1$, and that if the coefficient of the element $u \in M$ in $(\Z_p \otimes \nu)(\beta_w)$ is non-zero, then $w \preccurlyeq u$. A set $\mathcal B$ as in \cref{thm:Lyndon_basis} will be called a \emph{Lyndon heap} for the Lie algebra $\gr(A_\Gamma)$.

\subsection{Proof of the main result}

We need two preliminary lemmas.

\begin{lem}\label{lem:centreless_centre}
    If $A_\Gamma$ is a centreless RAAG, then 
    \[
        Z(N_{c,\Gamma}) = \gamma_c(N_{c,\Gamma}) \quad \text{and} \quad Z(\mathbf N_{c,\Gamma}) = \gamma_c(\mathbf N_{c,\Gamma})
    \]
\end{lem}
\begin{proof}
    We begin with the proof in the discrete case. Since $N_{c,\Gamma}$ is nilpotent of class at most $c$, it is immediate that 
    \[
        \gamma_c(N_{c,\Gamma}) \subseteq Z(N_{c,\Gamma}).
    \]
    Now suppose that $g \in N_{c,\Gamma} \smallsetminus \gamma_c(N_{c,\Gamma})$, and choose a preimage $\widehat{g}$ of $g$ in $A_\Gamma$. Then $\widehat g \in \gamma_i(A_\Gamma)$ for some $i < c$, and thus $x = \nu(\widehat g \gamma_{i+1}(A_\Gamma))$ is non-zero in $\Z\llangle \Gamma\rrangle$ by the injectivity of $\nu$. By \cref{prop:non-trivial-coefficient}, there is a vertex $v \in V(\Gamma)$ such that $[v,x] \neq 0$. But $v = \nu(s_v)$, and therefore 
    \[
        [s_v \gamma_2(A_\Gamma), \widehat g \gamma_{i+1}(A_\Gamma)] = [s_v,\widehat g] \gamma_{i+2}(A_\Gamma)
    \]
    is non-zero in $\gr(A_\Gamma)$. But then $[s_v,g]$ is also non-zero in $\gr(N_{c,\Gamma})$ (where we are abusing notation by considering $s_v$ as an element of $N_{c,\Gamma})$, since $i < c$. Hence, $g$ is not central. 
    
    The pro-$p$ case follows at once by taking the closures of $Z(N_{c,\Gamma})$ and $\gamma_c(N_{c,\Gamma})$. \qedhere
\end{proof}

\begin{lem}\label{lem:normaliser_reformulation}
    Let $G$ be a finitely generated residually $p$ group and let $\mathbf G$ be its pro-$p$ completion. The natural map
    \[
        \Out(G) \to \Out(\mathbf G)
    \]
    is injective if and only if $N_{\mathbf G}(G) = Z(\mathbf G)G$.
\end{lem}
\begin{proof}
    Let $x \in N_{\mathbf G}(G)$. Then the conjugation map $(\cdot)^x \colon \mathbf G \to \mathbf G$ restricts to an automorphism of $G$. If $\Out(G) \to \Out(\mathbf G)$ is injective, then $(\cdot)^x$ must restrict to an inner automorphism of $G$, so there is some $g \in G$ such that $(\cdot)^g = (\cdot)^x$ on $G$. But $G$ is dense in $\mathbf G$, so in fact $(\cdot)^g = (\cdot)^x$ on $\mathbf G$, which implies that $g\inv x \in Z(\mathbf G)$, and therefore that $g \in Z(\mathbf G)G$.

    Conversely, suppose that $N_{\mathbf G}(G) = Z(\mathbf G)G$, and $x \in \mathbf G$ is an element such that $\varphi \in \Aut(G)$ induces the inner automorphism of $\mathbf G$ given by conjugation by $x$. Then $x$ normalises $G$, so $x = zg$ for some $z \in Z(\mathbf G)$ and $g \in G$. But then $(\cdot)^x = (\cdot)^g$, implying that $\varphi \in \Inn(G)$. \qedhere
\end{proof}

We are now ready to prove the main result, which we restate in full generality. \cref{thm:p_inner_free_nilpotent} follows by setting $\Gamma$ to be a graph with no edges in the following theorem.

\begin{thm}\label{thm:appendix_main}
    Let $c \in \N$ and let $\Gamma$ be finite simplicial graph. The natural map
    \[
        \Out(N_{c,\Gamma}) \to \Out(\mathbf N_{c,\Gamma})
    \]
    is injective.
\end{thm}
\begin{proof}
    For ease of notation, we will denote $N_{c,\Gamma}$ by $N_c$. We will prove the claim by induction on $c$. When $c = 1$, the claim is trivial since $N_1$ is a free Abelian group. Assume $c > 1$ and that the result holds for free partially commutative nilpotent groups of class less than $c$. Let $\varphi \in \Aut(N_c)$ be such that $\bm\varphi \in \Inn(\mathbf N_c)$, where $\bm{\varphi} \colon \mathbf N_c \to \mathbf N_c$ is the automorphism of pro-$p$ completions induced by $\varphi$. Let $x \in \mathbf N_c$ be such that $\bm\varphi = (\cdot)^x$. Our goal is to prove that $x \in Z(\mathbf N_c) N_c$.

    \begin{claim}
        It suffices to assume that $A_\Gamma$ is centreless.
    \end{claim}
    \begin{proof}
        Let $\Lambda \leqslant \Gamma$ be the full clique subgraph such that $Z(A_\Gamma) = A_\Lambda$. Then
        \[
            A_\Gamma \cong A_\Lambda \times A_\Upsilon
        \]
        for some full subgraph $\Upsilon \leqslant \Gamma$ and, moreover, there are isomorphisms 
        \[
            N_{c,\Gamma} \cong A_\Lambda \times N_{c,\Upsilon} \quad \text{and} \quad \mathbf N_{c,\Gamma} \cong \mathbf A_\Lambda \times \mathbf N_{c,\Upsilon}.
        \]
        Write $x = ab$ for $a \in \mathbf A_\Lambda$ and $b \in \mathbf N_{c,\Upsilon}$. Then $(\cdot)^b \colon \mathbf N_{c,\Upsilon} \to \mathbf N_{c,\Upsilon}$ restrits to an automorphism of $N_{c,\Upsilon}$. Therefore $b = st$ for some $s \in Z(\mathbf N_{c,\Upsilon})$ and $t \in N_{c,\Upsilon}$, since $A_\Upsilon$ is centreless. But then $x = (as)t$, where $as$ is central, as desired. \renewcommand\qedsymbol{$\diamond$}\qedhere
    \end{proof}
    
    From now on, we assume that $A_\Gamma$ is centreless. Since $\overline x \in \mathbf N_c/\gamma_c(\mathbf N_c) \cong \mathbf N_{c-1}$ normalises $N_c/\gamma_c(N_c) \cong N_{c-1}$, by induction and by \cref{lem:normaliser_reformulation} we have that
    \[
        \overline x \in Z(\mathbf N_c/\gamma_c(\mathbf N_c)) \left( N_c/\gamma_c(N)\right) = \frac{\gamma_{c-1}(\mathbf N_c) N_c}{\gamma_c(\mathbf N_c)}.
    \]
    Note that we have used \cref{lem:centreless_centre} to conclude that that 
    \[
        Z(\mathbf N_c/\gamma_c(\mathbf N_c)) = \gamma_{c-1}(\mathbf N_c).
    \]
    This shows that $x \in \gamma_{c-1}(\mathbf N_c)N_c$, so there are elements $y \in \gamma_{c-1}(\mathbf N_c)$ and $g \in N_c$ such that $x = yg$.

    Fix a Lyndon basis $\mathcal B$ for $\gr(F)$ (with notation as in \cref{thm:Lyndon_basis}) and express $\overline y \in \gamma_{c-1}(\mathbf N_c)/\gamma_c(\mathbf N_c) \cong \Z_p^m$ as a sum $\overline y = \sum_{i=1}^m \lambda_i \beta_{w_i}$ for some $\lambda_i \in \Z_p$. Without loss of generality we assume that $w_1 \prec \dots \prec w_m$, where $\prec$ is the order on $M$. We are slightly abusing notation here by making the natural identification
    \[
        \gamma_{c-1}(\mathbf N_c)/\gamma_c(\mathbf N_c) \cong \gamma_{c-1}(\mathbf A_\Gamma)/\gamma_c(\mathbf A_\Gamma).
    \]

    \begin{claim}\label{claim:coeffs_in_Z}
        We have $\lambda_i \in \Z$ for each $i \in \{1, \dots, m\}$.
    \end{claim}
    \begin{proof}
        Suppose for a contradiction that not all the coefficients $\lambda_i$ lie in $\Z$. Let $i \in \{1, \dots, m\}$ be minimal such that $\lambda_i \in \Z_p \smallsetminus \Z$. Suppose that $w_i$ appears with non-zero coefficient in $(\Z_p \otimes \nu)(\beta_{w_j})$ for some $i < j$. By \cref{thm:Lyndon_basis}, we then have $w_j \prec w_i \prec w_j$, a contradiction. Hence, $w_i$ can only appear with non-zero coefficient in $(\Z_p \otimes \nu)(\beta_{w_j})$ if $j \leqslant i$, and if $j < i$ then the coefficient of $w_i$ lies in $\Z$ by the minimality assumption. Hence, the coefficient $\tau$ of $w_i$ in $(\Z_p \otimes \nu)(\overline y)$ lies in $\Z_p \smallsetminus \Z$. By \cref{prop:non-trivial-coefficient}, there is some $x_l \in X$ such that the coefficient of $w_i x_l$ in 
        \[
            [(\Z_p \otimes \nu)(\overline y), x_l]
        \]
        is $\tau \in \Z_p \smallsetminus \Z$.

        On the other hand, since $x$ normalises $N_c$, so must $y$, and therefore 
        \[
            [\overline y, s_l] \in \gamma_c(N_c) \subseteq \gamma_c(\mathbf N_c).
        \]
        But $[(\Z_p \otimes \nu)(\overline y), x_l]$ is the image of $[\overline y, s_l]\gamma_{c+1}(\mathbf F)$ under $\Z_p \otimes \nu$, and therefore $[(\Z_p \otimes \nu)(\overline y), x_l] \in \Z\llangle X\rrangle$, contradicting the conclusion of the previous paragraph. Therefore, $\lambda_i \in \Z$ for each $i \in \{1, \dots, m\}$. \renewcommand\qedsymbol{$\diamond$}\qedhere
    \end{proof}
    
    But \cref{claim:coeffs_in_Z} implies that 
    \[
        \overline y \in \gamma_{c-1}(N_c)/\gamma_c(N_c) \subset \gamma_{c-1}(\mathbf N_c)/\gamma_c(\mathbf N_c)
    \]
    or put differently that there are elements $z \in \gamma_c(\mathbf N_c) = Z(\mathbf N_c)$ and $y' \in \gamma_{c-1}(N_c)$ such that $y = zy'$. Then $x = gy'z$, which concludes the proof. \qedhere
\end{proof}

It is well known that $\mathcal T_p(\Z^n)$ is residually $p$ for all primes $p$. As an application of \cref{thm:p_inner_free_nilpotent}, we prove a corresponding statement for finitely generated free nilpotent groups, and more generally the groups $N_{c,\Gamma}$.

\begin{cor}\label{cor:res_p_free_nilp}
    Let $\Gamma$ be a finite simplicial graph and let $c \in \N$. The group $\mathcal T_p(N_{c,\Gamma})$ is residually $p$ for all primes $p$.
\end{cor}
\begin{proof}
    Since $\Gamma$ is fixed, we will abbreviate $N_{c,\Gamma}$ by $N_c$. As usual, we let $\mathbf N_c$ denote the pro-$p$ completion of $N_c$. By \cite[Theorem 5.6]{DixonEtAl_analyticPropBook}, $T_p(\mathbf N_c)$ is a pro-$p$ group, and in particular is residually $p$. There is a commuting diagram of short exact sequences
    \[
        \begin{tikzcd}
            1 \arrow[r] & T_p(\mathbf N_c) \arrow[r] \arrow[d] & \Aut(\mathbf N_c) \arrow[r] \arrow[d] & {\Aut(\mathbf N_c/[\mathbf N_c,\mathbf N_c]\mathbf N_c^p)} \arrow[d, Rightarrow, no head] \\
            1 \arrow[r] & \mathcal T_p(\mathbf N_c) \arrow[r]  & \Out(\mathbf N_c) \arrow[r] & {\Out(\mathbf N_c/[\mathbf N_c,\mathbf N_c]\mathbf N_c^p)} \nospacepunct{,}
        \end{tikzcd}
    \]
    where the central vertical map is surjective by definition, and the rightmost vertical map is a canonical isomorphism since $\mathbf N_c/[\mathbf N_c,\mathbf N_c]\mathbf N_c^p$ is Abelian (if $p = 2$, then one must replace $p$ with $p^2$ in the above diagram). By the Four Lemma, it follows that the leftmost vertical map is onto, and therefore $\mathcal T_p(\mathbf N_c)$ is pro-$p$, being the quotient of a pro-$p$ group. There is a commuting square
    \[
        \begin{tikzcd}
            \mathcal T_p(N_c) \arrow[r] \arrow[d] & \Out(N_c) \arrow[d] \\
            \mathcal T_p(\mathbf N_c) \arrow[r] & \Out(\mathbf N_c) \nospacepunct{,}
        \end{tikzcd}
    \]
    where the horizontal maps are inclusions by definition, and the right vertical map is an inclusion by \cref{thm:p_inner_free_nilpotent}. Hence, the left vertical map is an inclusion and thus $\mathcal T_p(N_c)$ is residually $p$, being a subgroup of a pro-$p$ group. \qedhere
\end{proof}

\bibliography{bib}

@article{AbertBergeronFraczykGaboriau_torsion,
  author   = {Abert, Miklos and Bergeron, Nicolas and Fr{\k{a}}czyk, Miko{\l}aj and Gaboriau, Damien},
  title    = {On homology torsion growth},
  journal  = {J. Eur. Math. Soc. (JEMS)},
  fjournal = {Journal of the European Mathematical Society (JEMS)},
  volume   = {27},
  year     = {2025},
  number   = {6},
  pages    = {2293--2357},
  issn     = {1435-9855,1435-9863},
  mrclass  = {20E26 (20F36 20F69 20G25)},
  mrnumber = {4889244},
  doi      = {10.4171/jems/1411},
  url      = {https://doi.org/10.4171/jems/1411}
}

@article{Andreadakis1965,
  author     = {Andreadakis, S.},
  title      = {On the automorphisms of free groups and free nilpotent groups},
  journal    = {Proc. London Math. Soc. (3)},
  fjournal   = {Proceedings of the London Mathematical Society. Third Series},
  volume     = {15},
  year       = {1965},
  pages      = {239--268},
  issn       = {0024-6115,1460-244X},
  mrclass    = {20.22 (20.10)},
  mrnumber   = {188307},
  mrreviewer = {W.\ Moser},
  doi        = {10.1112/plms/s3-15.1.239},
  url        = {https://doi.org/10.1112/plms/s3-15.1.239}
}

@article{Asada1995,
  author     = {Asada, Mamoru},
  title      = {Two properties of the filtration of the outer automorphism
                groups of certain groups},
  journal    = {Math. Z.},
  fjournal   = {Mathematische Zeitschrift},
  volume     = {218},
  year       = {1995},
  number     = {1},
  pages      = {123--133},
  issn       = {0025-5874,1432-1823},
  mrclass    = {20F28 (12F99 20E18)},
  mrnumber   = {1312581},
  mrreviewer = {Marcus\ du Sautoy},
  doi        = {10.1007/BF02571892},
  url        = {https://doi.org/10.1007/BF02571892}
}

@incollection{Atiyah_OGL2,
  author     = {Atiyah, M. F.},
  title      = {Elliptic operators, discrete groups and von {N}eumann
                algebras},
  booktitle  = {Colloque ``{A}nalyse et {T}opologie'' en l'{H}onneur de
                {H}enri {C}artan ({O}rsay, 1974)},
  series     = {Ast\'{e}risque},
  volume     = {No. 32-33},
  pages      = {43--72},
  publisher  = {Soc. Math. France, Paris},
  year       = {1976},
  mrclass    = {58G10 (22E45 46L10)},
  mrnumber   = {420729},
  mrreviewer = {R.\ D.\ Moyer}
}

@article{Austin2013,
  author     = {Austin, Tim},
  title      = {Rational group ring elements with kernels having irrational
                dimension},
  journal    = {Proc. Lond. Math. Soc. (3)},
  fjournal   = {Proceedings of the London Mathematical Society. Third Series},
  volume     = {107},
  year       = {2013},
  number     = {6},
  pages      = {1424--1448},
  issn       = {0024-6115,1460-244X},
  mrclass    = {20F65 (05C25 58J22)},
  mrnumber   = {3149852},
  mrreviewer = {Qihui\ Li},
  doi        = {10.1112/plms/pdt029},
  url        = {https://doi.org/10.1112/plms/pdt029}
}

@incollection{BassLubotzky1994,
  author     = {Bass, Hyman and Lubotzky, Alexander},
  title      = {Linear-central filtrations on groups},
  booktitle  = {The mathematical legacy of {W}ilhelm {M}agnus: groups,
                geometry and special functions ({B}rooklyn, {NY}, 1992)},
  series     = {Contemp. Math.},
  volume     = {169},
  pages      = {45--98},
  publisher  = {Amer. Math. Soc., Providence, RI},
  year       = {1994},
  isbn       = {0-8218-5156-X},
  mrclass    = {20F14 (16U60 20E15)},
  mrnumber   = {1292897},
  mrreviewer = {Dimitrios\ Varsos},
  doi        = {10.1090/conm/169/01651},
  url        = {https://doi.org/10.1090/conm/169/01651}
}

@article{Berberian_vonNeumannOre,
  author     = {Berberian, S. K.},
  title      = {The maximal ring of quotients of a finite von {N}eumann
                algebra},
  journal    = {Rocky Mountain J. Math.},
  fjournal   = {The Rocky Mountain Journal of Mathematics},
  volume     = {12},
  year       = {1982},
  number     = {1},
  pages      = {149--164},
  issn       = {0035-7596,1945-3795},
  mrclass    = {16A08 (16A30 46L10)},
  mrnumber   = {649748},
  mrreviewer = {David\ Handelman},
  doi        = {10.1216/RMJ-1982-12-1-149},
  url        = {https://doi.org/10.1216/RMJ-1982-12-1-149}
}

@article{BigelowBudney2001,
  author     = {Bigelow, Stephen J. and Budney, Ryan D.},
  title      = {The mapping class group of a genus two surface is linear},
  journal    = {Algebr. Geom. Topol.},
  fjournal   = {Algebraic \& Geometric Topology},
  volume     = {1},
  year       = {2001},
  pages      = {699--708},
  issn       = {1472-2747,1472-2739},
  mrclass    = {20F36 (20C15 57M07)},
  mrnumber   = {1875613},
  mrreviewer = {Luis\ Paris},
  doi        = {10.2140/agt.2001.1.699},
  url        = {https://doi.org/10.2140/agt.2001.1.699}
}

@article{Casals-RuizPintonelloZalesskii2025,
  author     = {Casals-Ruiz, Montserrat and Pintonello, Matteo and Zalesskii,
                Pavel},
  title      = {Pro-{$\mathcal{C}$} {RAAG}s},
  journal    = {J. Algebra},
  fjournal   = {Journal of Algebra},
  volume     = {664},
  year       = {2025},
  pages      = {177--208},
  issn       = {0021-8693,1090-266X},
  mrclass    = {20E08 (20E06 20E18)},
  mrnumber   = {4812314},
  mrreviewer = {Matteo\ Vannacci},
  doi        = {10.1016/j.jalgebra.2024.09.030},
  url        = {https://doi.org/10.1016/j.jalgebra.2024.09.030}
}

@article{CharneyVogtmann2009,
  author     = {Charney, Ruth and Vogtmann, Karen},
  title      = {Finiteness properties of automorphism groups of right-angled
                {A}rtin groups},
  journal    = {Bull. Lond. Math. Soc.},
  fjournal   = {Bulletin of the London Mathematical Society},
  volume     = {41},
  year       = {2009},
  number     = {1},
  pages      = {94--102},
  issn       = {0024-6093,1469-2120},
  mrclass    = {20F36},
  mrnumber   = {2481994},
  mrreviewer = {Volker\ Gebhardt},
  doi        = {10.1112/blms/bdn108},
  url        = {https://doi.org/10.1112/blms/bdn108}
}

@article{DahmaniFrancavigliaMartinoTouikan_ConjOutF3_2025,
  author     = {Dahmani, Fran\c cois and Francaviglia, Stefano and Martino, Armando and Touikan, Nicholas},
  title      = {The conjugacy problem for {${\rm Out}(F_3)$}},
  journal    = {Forum Math. Sigma},
  fjournal   = {Forum of Mathematics. Sigma},
  volume     = {13},
  year       = {2025},
  pages      = {Paper No. e41, 21},
  issn       = {2050-5094},
  mrclass    = {20F65 (20E05 20E36 20F10 20F67)},
  mrnumber   = {4863720},
  mrreviewer = {Michael\ Hull},
  doi        = {10.1017/fms.2025.3},
  url        = {https://doi.org/10.1017/fms.2025.3}
}

@book{DixonEtAl_analyticPropBook,
  author     = {Dixon, J. D. and du Sautoy, M. P. F. and Mann, A. and Segal,
                D.},
  title      = {Analytic pro-{$p$}-groups},
  series     = {London Mathematical Society Lecture Note Series},
  volume     = {157},
  publisher  = {Cambridge University Press, Cambridge},
  year       = {1991},
  pages      = {x+251},
  isbn       = {0-521-39580-1},
  mrclass    = {20E18 (20G30)},
  mrnumber   = {1152800},
  mrreviewer = {Andy\ R.\ Magid}
}

@article{DuchampKrob_RAAGsRTFN,
  author   = {Duchamp, G\'erard and Krob, Daniel},
  title    = {The lower central series of the free partially commutative
              group},
  journal  = {Semigroup Forum},
  fjournal = {Semigroup Forum},
  volume   = {45},
  year     = {1992},
  number   = {3},
  pages    = {385--394},
  issn     = {0037-1912},
  mrclass  = {20F14},
  mrnumber = {1179860},
  doi      = {10.1007/BF03025778},
  url      = {https://doi-org.ezproxy-prd.bodleian.ox.ac.uk/10.1007/BF03025778}
}

@article{Endimioni2002,
  author     = {Endimioni, G.},
  title      = {Pointwise inner automorphisms in a free nilpotent group},
  journal    = {Q. J. Math.},
  fjournal   = {The Quarterly Journal of Mathematics},
  volume     = {53},
  year       = {2002},
  number     = {4},
  pages      = {397--402},
  issn       = {0033-5606,1464-3847},
  mrclass    = {20E36 (20F18)},
  mrnumber   = {1949151},
  mrreviewer = {Martyn\ R.\ Dixon},
  doi        = {10.1093/qjmath/53.4.397},
  url        = {https://doi.org/10.1093/qjmath/53.4.397}
}

@article{Epstein66,
  author     = {Epstein, D. B. A.},
  title      = {Curves on {$2$}-manifolds and isotopies},
  journal    = {Acta Math.},
  fjournal   = {Acta Mathematica},
  volume     = {115},
  year       = {1966},
  pages      = {83--107},
  issn       = {0001-5962,1871-2509},
  mrclass    = {57.20},
  mrnumber   = {214087},
  mrreviewer = {P.\ Dedecker},
  doi        = {10.1007/BF02392203},
  url        = {https://doi.org/10.1007/BF02392203}
}

@incollection{FarkasLinnell2006,
  author     = {Farkas, Daniel R. and Linnell, Peter A.},
  title      = {Congruence subgroups and the {A}tiyah conjecture},
  booktitle  = {Groups, rings and algebras},
  series     = {Contemp. Math.},
  volume     = {420},
  pages      = {89--102},
  publisher  = {Amer. Math. Soc., Providence, RI},
  year       = {2006},
  isbn       = {978-0-8218-3904-1; 0-8218-3904-7},
  mrclass    = {16S34 (20C07 46L10)},
  mrnumber   = {2279234},
  mrreviewer = {Kenneth\ A.\ Brown},
  doi        = {10.1090/conm/420/07970},
  url        = {https://doi.org/10.1090/conm/420/07970}
}

@misc{fishersanchez_divrings,
  title         = {Division rings for group algebras of virtually compact special groups and $3$-manifold groups},
  author        = {Sam P. Fisher and Pablo Sánchez-Peralta},
  year          = {2023},
  note          = {{\tt arXiv:2303.08165} (to appear in \emph{J. Comb. Algebra})},
  eprint        = {2303.08165},
  archiveprefix = {arXiv},
  primaryclass  = {math.GR}
}

@article{FormanekProcesi1992,
  author     = {Formanek, Edward and Procesi, Claudio},
  title      = {The automorphism group of a free group is not linear},
  journal    = {J. Algebra},
  fjournal   = {Journal of Algebra},
  volume     = {149},
  year       = {1992},
  number     = {2},
  pages      = {494--499},
  issn       = {0021-8693,1090-266X},
  mrclass    = {20F28},
  mrnumber   = {1172442},
  mrreviewer = {Alexander\ Lubotzky},
  doi        = {10.1016/0021-8693(92)90029-L},
  url        = {https://doi.org/10.1016/0021-8693(92)90029-L}
}

@article{FriedlLuck_euler,
  author     = {Friedl, Stefan and L\"{u}ck, Wolfgang},
  title      = {{$L^2$}-{E}uler characteristics and the {T}hurston norm},
  journal    = {Proc. Lond. Math. Soc. (3)},
  fjournal   = {Proceedings of the London Mathematical Society. Third Series},
  volume     = {118},
  year       = {2019},
  number     = {4},
  pages      = {857--900},
  issn       = {0024-6115},
  mrclass    = {57M27 (22D25 58J52)},
  mrnumber   = {3938714},
  mrreviewer = {Nikhil Savale},
  doi        = {10.1112/plms.12202},
  url        = {https://doi-org.ezproxy-prd.bodleian.ox.ac.uk/10.1112/plms.12202}
}

@article{GaboriauNous_topDimL2_2021,
  author     = {Gaboriau, Damien and No{\^u}s, Camille},
  title      = {On the top-dimensional {$\ell^2$}-{B}etti numbers},
  journal    = {Ann. Fac. Sci. Toulouse Math. (6)},
  fjournal   = {Annales de la Facult\'e{} des Sciences de Toulouse.
                Math\'ematiques. S\'erie 6},
  volume     = {30},
  year       = {2021},
  number     = {5},
  pages      = {1121--1137},
  issn       = {0240-2963,2258-7519},
  mrclass    = {57K35 (19K56 20E15 20F28 37A20)},
  mrnumber   = {4401387},
  mrreviewer = {Kevin\ D.\ Schreve},
  doi        = {10.5802/afst.1695},
  url        = {https://doi.org/10.5802/afst.1695}
}

@article{Grabowski2014,
  author     = {Grabowski, \L ukasz},
  title      = {On {T}uring dynamical systems and the {A}tiyah problem},
  journal    = {Invent. Math.},
  fjournal   = {Inventiones Mathematicae},
  volume     = {198},
  year       = {2014},
  number     = {1},
  pages      = {27--69},
  issn       = {0020-9910,1432-1297},
  mrclass    = {20F65 (20L05 37A30)},
  mrnumber   = {3260857},
  mrreviewer = {Alain\ Valette},
  doi        = {10.1007/s00222-013-0497-5},
  url        = {https://doi.org/10.1007/s00222-013-0497-5}
}

@article{Grossman1974,
  author     = {Grossman, Edna K.},
  title      = {On the residual finiteness of certain mapping class groups},
  journal    = {J. London Math. Soc. (2)},
  fjournal   = {Journal of the London Mathematical Society. Second Series},
  volume     = {9},
  year       = {1974/75},
  pages      = {160--164},
  issn       = {0024-6107,1469-7750},
  mrclass    = {57A05 (20E25)},
  mrnumber   = {405423},
  mrreviewer = {J.\ S.\ Birman},
  doi        = {10.1112/jlms/s2-9.1.160},
  url        = {https://doi.org/10.1112/jlms/s2-9.1.160}
}

@incollection{Grunewald_conjugacy,
  author     = {Grunewald, Fritz J.},
  title      = {Solution of the conjugacy problem in certain arithmetic
                groups},
  booktitle  = {Word problems, {II} ({C}onf. on {D}ecision {P}roblems in
                {A}lgebra, {O}xford, 1976)},
  series     = {Stud. Logic Found. Math.},
  volume     = {95},
  pages      = {101--139},
  publisher  = {North-Holland, Amsterdam-New York},
  year       = {1980},
  isbn       = {0-444-85343-X},
  mrclass    = {20G25},
  mrnumber   = {579942},
  mrreviewer = {R.\ C.\ Lyndon}
}

@article{Harris_padicdescent_1979,
  author   = {Harris, Michael},
  title    = {{{\(p\)}}-adic representations arising from descent on {Abelian} varieties},
  fjournal = {Compositio Mathematica},
  journal  = {Compos. Math.},
  issn     = {0010-437X},
  volume   = {39},
  pages    = {177--245},
  year     = {1979},
  language = {English},
  url      = {https://eudml.org/doc/89418},
  zbmath   = {3648898},
  zbl      = {0417.14034}
}

@article{Hemion1979,
  author     = {Hemion, Geoffrey},
  title      = {On the classification of homeomorphisms of {$2$}-manifolds and
                the classification of {$3$}-manifolds},
  journal    = {Acta Math.},
  fjournal   = {Acta Mathematica},
  volume     = {142},
  year       = {1979},
  number     = {1-2},
  pages      = {123--155},
  issn       = {0001-5962,1871-2509},
  mrclass    = {57N05 (57N10)},
  mrnumber   = {512214},
  mrreviewer = {J.\ S.\ Birman},
  doi        = {10.1007/BF02395059},
  url        = {https://doi.org/10.1007/BF02395059}
}

@incollection{Hensel_Survey,
  author    = {Hensel, Sebastian},
  title     = {A primer on handlebody groups},
  booktitle = {Handbook of group actions. {V}},
  series    = {Adv. Lect. Math. (ALM)},
  volume    = {48},
  pages     = {143--177},
  publisher = {Int. Press, Somerville, MA},
  year      = {[2020] \copyright 2020},
  isbn      = {978-1-57146-390-6},
  mrclass   = {20F65 (57M07 57M60)},
  mrnumber  = {4237892}
}

@misc{Ivanova2002,
  title         = {On the conjugacy separability in the class of finite $p$-groups of finitely generated nilpotent groups},
  author        = {Elena A Ivanova},
  year          = {2002},
  eprint        = {0408393},
  archiveprefix = {arXiv},
  primaryclass  = {math.GR}
}

@article{JaikinLinton_coherence,
  author   = {Jaikin-Zapirain, Andrei and Linton, Marco},
  title    = {On the coherence of one-relator groups and their group
              algebras},
  journal  = {Ann. of Math. (2)},
  fjournal = {Annals of Mathematics. Second Series},
  volume   = {201},
  year     = {2025},
  number   = {3},
  pages    = {909--959},
  issn     = {0003-486X},
  mrclass  = {20E07 (20E08 20J05)},
  mrnumber = {4899802},
  doi      = {10.4007/annals.2025.201.3.4},
  url      = {https://doi-org.ezproxy-prd.bodleian.ox.ac.uk/10.4007/annals.2025.201.3.4}
}

@misc{JaikinLintonSanchez_OneRelProd,
  title         = {Group pairs, coherence and Farrell--Jones Conjecture for {$K_0$}},
  author        = {Andrei Jaikin-Zapirain and Marco Linton and Pablo Sánchez-Peralta},
  year          = {2025},
  eprint        = {2510.23518},
  note          = {{\tt arXiv:2510.23518}},
  archiveprefix = {arXiv},
  primaryclass  = {math.GR},
  url           = {https://arxiv.org/abs/2510.23518}
}

@article{JaikinLopezStrongAtiyah2020,
  author     = {Jaikin-Zapirain, Andrei and L{\'{o}}pez-{\'{A}}lvarez, Diego},
  title      = {The strong {A}tiyah and {L}\"{u}ck approximation conjectures for
                one-relator groups},
  journal    = {Math. Ann.},
  fjournal   = {Mathematische Annalen},
  volume     = {376},
  year       = {2020},
  number     = {3-4},
  pages      = {1741--1793},
  issn       = {0025-5831},
  mrclass    = {20F05 (16K40 16S34)},
  mrnumber   = {4081128},
  mrreviewer = {Shoumin Liu},
  doi        = {10.1007/s00208-019-01926-0},
  url        = {https://doi.org/10.1007/s00208-019-01926-0}
}

@article{JaikinZapirain_basechange,
  author   = {Jaikin-Zapirain, Andrei},
  title    = {The base change in the {A}tiyah and the {L}\"uck approximation
              conjectures},
  journal  = {Geom. Funct. Anal.},
  fjournal = {Geometric and Functional Analysis},
  volume   = {29},
  year     = {2019},
  number   = {2},
  pages    = {464--538},
  issn     = {1016-443X,1420-8970},
  mrclass  = {20C07 (16E50 16S34 46L10 47A58)},
  mrnumber = {3945838},
  doi      = {10.1007/s00039-019-00487-3},
  url      = {https://doi.org/10.1007/s00039-019-00487-3}
}

@incollection{Johnson_survey,
  author     = {Johnson, Dennis},
  title      = {A survey of the {T}orelli group},
  booktitle  = {Low-dimensional topology ({S}an {F}rancisco, {C}alif., 1981)},
  series     = {Contemp. Math.},
  volume     = {20},
  pages      = {165--179},
  publisher  = {Amer. Math. Soc., Providence, RI},
  year       = {1983},
  isbn       = {0-8218-5016-4},
  mrclass    = {57N05 (14H15 32G15 57-02)},
  mrnumber   = {718141},
  mrreviewer = {J.\ S.\ Birman},
  doi        = {10.1090/conm/020/718141},
  url        = {https://doi.org/10.1090/conm/020/718141}
}

@article{KalubaKielakNowak_T,
  author     = {Kaluba, Marek and Kielak, Dawid and Nowak, Piotr W.},
  title      = {On property ({T}) for {$\rm{Aut}(F_n)$} and
                {$\rm{SL}_n(\Bbb {Z})$}},
  journal    = {Ann. of Math. (2)},
  fjournal   = {Annals of Mathematics. Second Series},
  volume     = {193},
  year       = {2021},
  number     = {2},
  pages      = {539--562},
  issn       = {0003-486X,1939-8980},
  mrclass    = {22D55 (20F28)},
  mrnumber   = {4224715},
  mrreviewer = {Marco\ Trombetti},
  doi        = {10.4007/annals.2021.193.2.3},
  url        = {https://doi.org/10.4007/annals.2021.193.2.3}
}

@article{Kazhdan1967,
  author     = {Ka{\v z}dan, D. A.},
  title      = {On the connection of the dual space of a group with the
                structure of its closed subgroups},
  journal    = {Funkcional. Anal. i Prilo\v zen.},
  fjournal   = {Akademija Nauk SSSR. Funkcional\cprime nyi Analiz i ego
                Prilo\v zenija},
  volume     = {1},
  year       = {1967},
  pages      = {71--74},
  issn       = {0374-1990},
  mrclass    = {22.20},
  mrnumber   = {209390},
  mrreviewer = {K.\ Strambach}
}

@misc{KielakLinton_3mfldAtiyah,
  title         = {The Atiyah conjecture for three-manifold groups},
  author        = {Kielak, Dawid and Linton, Marco},
  year          = {2023},
  note          = {{\tt arXiv:2303.15907}},
  eprint        = {2303.15907},
  archiveprefix = {arXiv},
  primaryclass  = {math.GT}
}

@incollection{KrophollerLinnellLueck2009,
  author     = {Kropholler, Peter and Linnell, Peter and L\"uck, Wolfgang},
  title      = {Groups of small homological dimension and the {A}tiyah
                conjecture},
  booktitle  = {Geometric and cohomological methods in group theory},
  series     = {London Math. Soc. Lecture Note Ser.},
  volume     = {358},
  pages      = {272--277},
  publisher  = {Cambridge Univ. Press, Cambridge},
  year       = {2009},
  isbn       = {978-0-521-75724-9},
  mrclass    = {20J06 (18G20 46L99)},
  mrnumber   = {2605183},
  mrreviewer = {Brita\ E. A. Nucinkis}
}

@article{KrophollerLinnellMoody_Ore,
  author     = {Kropholler, P. H. and Linnell, P. A. and Moody, J. A.},
  title      = {Applications of a new {$K$}-theoretic theorem to soluble group
                rings},
  journal    = {Proc. Amer. Math. Soc.},
  fjournal   = {Proceedings of the American Mathematical Society},
  volume     = {104},
  year       = {1988},
  number     = {3},
  pages      = {675--684},
  issn       = {0002-9939},
  mrclass    = {16A27 (16A08 16A34)},
  mrnumber   = {964842},
  mrreviewer = {L. N. Vaserstein},
  doi        = {10.2307/2046771},
  url        = {https://doi-org.ezproxy-prd.bodleian.ox.ac.uk/10.2307/2046771}
}

@incollection{Lalonde_basesLyndon93,
  author     = {Lalonde, Pierre},
  title      = {Bases de {L}yndon des alg\`ebres de {L}ie libres partiellement
                commutatives},
  note       = {Conference on Formal Power Series and Algebraic Combinatorics
                (Bordeaux, 1991)},
  journal    = {Theoret. Comput. Sci.},
  fjournal   = {Theoretical Computer Science},
  volume     = {117},
  year       = {1993},
  number     = {1-2},
  pages      = {217--226},
  issn       = {0304-3975,1879-2294},
  mrclass    = {17B01 (17B60)},
  mrnumber   = {1235180},
  mrreviewer = {Vesselin\ Drensky},
  doi        = {10.1016/0304-3975(93)90315-K},
  url        = {https://doi.org/10.1016/0304-3975(93)90315-K}
}

@article{Lazard1954,
  author     = {Lazard, Michel},
  title      = {Sur les groupes nilpotents et les anneaux de {L}ie},
  journal    = {Ann. Sci. \'Ecole Norm. Sup. (3)},
  fjournal   = {Annales Scientifiques de l'\'Ecole Normale Sup\'erieure.
                Troisi\`eme S\'erie},
  volume     = {71},
  year       = {1954},
  pages      = {101--190},
  issn       = {0012-9593},
  mrclass    = {20.0X},
  mrnumber   = {88496},
  mrreviewer = {B.\ H.\ Neumann},
  url        = {http://www.numdam.org/item?id=ASENS_1954_3_71_2_101_0}
}

@article{Lazard1965,
  author   = {Lazard, Michel},
  title    = {Groupes analytiques {$p$}-adiques},
  journal  = {Inst. Hautes \'Etudes Sci. Publ. Math.},
  fjournal = {Institut des Hautes \'Etudes Scientifiques. Publications
              Math\'ematiques},
  number   = {26},
  year     = {1965},
  pages    = {389--603},
  issn     = {0073-8301,1618-1913},
  mrclass  = {14.50},
  mrnumber = {209286},
  url      = {http://www.numdam.org/item?id=PMIHES_1965__26__389_0}
}

@article{LinnellDivRings93,
  author     = {Linnell, Peter A.},
  title      = {Division rings and group von {N}eumann algebras},
  journal    = {Forum Math.},
  fjournal   = {Forum Mathematicum},
  volume     = {5},
  year       = {1993},
  number     = {6},
  pages      = {561--576},
  issn       = {0933-7741},
  mrclass    = {20C07 (16K40 16S35 22D25 46L10)},
  mrnumber   = {1242889},
  mrreviewer = {Alain Valette},
  doi        = {10.1515/form.1993.5.561},
  url        = {https://ezproxy-prd.bodleian.ox.ac.uk:2102/10.1515/form.1993.5.561}
}

@article{LinnellSchick_AtiyahExt,
  author     = {Linnell, Peter and Schick, Thomas},
  title      = {Finite group extensions and the {A}tiyah conjecture},
  journal    = {J. Amer. Math. Soc.},
  fjournal   = {Journal of the American Mathematical Society},
  volume     = {20},
  year       = {2007},
  number     = {4},
  pages      = {1003--1051},
  issn       = {0894-0347},
  mrclass    = {58J22 (16S34 55N25 57R20)},
  mrnumber   = {2328714},
  mrreviewer = {Piotr W. Nowak},
  doi        = {10.1090/S0894-0347-07-00561-9},
  url        = {https://doi.org/10.1090/S0894-0347-07-00561-9}
}

@book{Luck02,
  author     = {L{\"u}ck, Wolfgang},
  title      = {{$L\sp 2$}-invariants: theory and applications to geometry and
                {$K$}-theory},
  publisher  = {Springer-Verlag},
  address    = {Berlin},
  year       = {2002},
  pages      = {xvi+595},
  isbn       = {3-540-43566-2},
  mrclass    = {58J22 (19K56 46L80 57Q10 57R20 58J52)},
  mrnumber   = {2003m:58033},
  mrreviewer = {Thomas Schick}
}

@article{Magnus_freegrpsRTFN,
  author   = {Magnus, Wilhelm},
  title    = {Beziehungen zwischen {G}ruppen und {I}dealen in einem
              speziellen {R}ing},
  journal  = {Math. Ann.},
  fjournal = {Mathematische Annalen},
  volume   = {111},
  year     = {1935},
  number   = {1},
  pages    = {259--280},
  issn     = {0025-5831},
  mrclass  = {DML},
  mrnumber = {1512992},
  doi      = {10.1007/BF01472217},
  url      = {https://doi-org.ezproxy-prd.bodleian.ox.ac.uk/10.1007/BF01472217}
}

@article{Malcev_series,
  author     = {Mal{'}cev, Anatoli\u\i I.},
  title      = {On the embedding of group algebras in division algebras},
  journal    = {Doklady Akad. Nauk SSSR (N.S.)},
  volume     = {60},
  year       = {1948},
  pages      = {1499--1501},
  mrclass    = {09.1X},
  mrnumber   = {0025457},
  mrreviewer = {I. Kaplansky}
}

@article{Malcolmson_SkewFields,
  author     = {Malcolmson, Peter},
  title      = {Determining homomorphisms to skew fields},
  journal    = {J. Algebra},
  fjournal   = {Journal of Algebra},
  volume     = {64},
  year       = {1980},
  number     = {2},
  pages      = {399--413},
  issn       = {0021-8693},
  mrclass    = {16A08},
  mrnumber   = {0579068},
  mrreviewer = {P. M. Cohn},
  doi        = {10.1016/0021-8693(80)90153-2},
  url        = {https://doi.org/10.1016/0021-8693(80)90153-2}
}

@book{McConnellRobsonNNR,
  author    = {McConnell, J. C. and Robson, J. C.},
  title     = {Noncommutative {N}oetherian rings},
  series    = {Graduate Studies in Mathematics},
  volume    = {30},
  edition   = {Revised},
  note      = {With the cooperation of L. W. Small},
  publisher = {American Mathematical Society, Providence, RI},
  year      = {2001},
  pages     = {xx+636},
  isbn      = {0-8218-2169-5},
  mrclass   = {16P40 (16-02)},
  mrnumber  = {1811901},
  doi       = {10.1090/gsm/030},
  url       = {https://doi.org/10.1090/gsm/030}
}

@article{McCool,
  author     = {McCool, James},
  title      = {A faithful polynomial representation of {${\rm Out}\,F_3$}},
  journal    = {Math. Proc. Cambridge Philos. Soc.},
  fjournal   = {Mathematical Proceedings of the Cambridge Philosophical
                Society},
  volume     = {106},
  year       = {1989},
  number     = {2},
  pages      = {207--213},
  issn       = {0305-0041,1469-8064},
  mrclass    = {20F28},
  mrnumber   = {1002533},
  mrreviewer = {W.\ Metzler and Cynthia\ Hog-Angeloni},
  doi        = {10.1017/S0305004100078026},
  url        = {https://doi.org/10.1017/S0305004100078026}
}

@article{Minasyan2012,
  author     = {Minasyan, Ashot},
  title      = {Hereditary conjugacy separability of right-angled {A}rtin
                groups and its applications},
  journal    = {Groups Geom. Dyn.},
  fjournal   = {Groups, Geometry, and Dynamics},
  volume     = {6},
  year       = {2012},
  number     = {2},
  pages      = {335--388},
  issn       = {1661-7207,1661-7215},
  mrclass    = {20F36 (20E26 20F10 20F55 20F67)},
  mrnumber   = {2914863},
  mrreviewer = {Ian\ M.\ Chiswell},
  doi        = {10.4171/GGD/160},
  url        = {https://doi.org/10.4171/GGD/160}
}

@article{Neumann_completedGpAlgDomain,
  author     = {Neumann, Andreas},
  title      = {Completed group algebras without zero divisors},
  journal    = {Arch. Math. (Basel)},
  fjournal   = {Archiv der Mathematik},
  volume     = {51},
  year       = {1988},
  number     = {6},
  pages      = {496--499},
  issn       = {0003-889X,1420-8938},
  mrclass    = {20C05 (16A26)},
  mrnumber   = {973723},
  mrreviewer = {Alexander\ Lubotzky},
  doi        = {10.1007/BF01261969},
  url        = {https://doi.org/10.1007/BF01261969}
}

@article{Neumann_series,
  author     = {Neumann, Bernhard H.},
  title      = {On ordered division rings},
  journal    = {Trans. Amer. Math. Soc.},
  fjournal   = {Transactions of the American Mathematical Society},
  volume     = {66},
  year       = {1949},
  pages      = {202--252},
  issn       = {0002-9947},
  mrclass    = {09.1X},
  mrnumber   = {32593},
  mrreviewer = {R. Moufang},
  doi        = {10.2307/1990552},
  url        = {https://doi-org.ezproxy-prd.bodleian.ox.ac.uk/10.2307/1990552}
}

@misc{Nitsche_AutF4propT_2023,
  author    = {Nitsche, Martin},
  title     = {Computer proofs for Property ({T}), and SDP duality},
  note      = {{\tt \href{https://arxiv.org/abs/2009.05134}{arXiv:2009.05134}}},
  publisher = {arXiv},
  year      = {2023},
  copyright = {arXiv.org perpetual, non-exclusive license}
}

@article{Paris2009,
  author     = {Paris, Luis},
  title      = {Residual {$p$} properties of mapping class groups and surface
                groups},
  journal    = {Trans. Amer. Math. Soc.},
  fjournal   = {Transactions of the American Mathematical Society},
  volume     = {361},
  year       = {2009},
  number     = {5},
  pages      = {2487--2507},
  issn       = {0002-9947,1088-6850},
  mrclass    = {20F38 (20E26 20F14 20F34 57M07)},
  mrnumber   = {2471926},
  mrreviewer = {Jos\'e\ Burillo},
  doi        = {10.1090/S0002-9947-08-04573-X},
  url        = {https://doi.org/10.1090/S0002-9947-08-04573-X}
}

@article{PichotSchickZuk,
  author     = {Pichot, Mika\"el and Schick, Thomas and Zuk, Andrzej},
  title      = {Closed manifolds with transcendental {$L^2$}-{B}etti numbers},
  journal    = {J. Lond. Math. Soc. (2)},
  fjournal   = {Journal of the London Mathematical Society. Second Series},
  volume     = {92},
  year       = {2015},
  number     = {2},
  pages      = {371--392},
  issn       = {0024-6107,1469-7750},
  mrclass    = {20F65 (57R19)},
  mrnumber   = {3404029},
  mrreviewer = {Kevin\ D.\ Schreve},
  doi        = {10.1112/jlms/jdv026},
  url        = {https://doi.org/10.1112/jlms/jdv026}
}

@article{SanchezPeralta2025,
  author     = {S\'anchez-Peralta, Pablo},
  title      = {Universal localizations, {A}tiyah conjectures and graphs of
                groups},
  journal    = {Geom. Funct. Anal.},
  fjournal   = {Geometric and Functional Analysis},
  volume     = {35},
  year       = {2025},
  number     = {3},
  pages      = {842--876},
  issn       = {1016-443X,1420-8970},
  mrclass    = {20F65 (16S85 19A49 19B99 20C07)},
  mrnumber   = {4915744},
  mrreviewer = {Francesco\ G.\ Russo},
  doi        = {10.1007/s00039-025-00710-4},
  url        = {https://doi.org/10.1007/s00039-025-00710-4}
}

@article{Schreve_AtiyahVCS,
  author     = {Schreve, Kevin},
  title      = {The strong {A}tiyah conjecture for virtually cocompact special
                groups},
  journal    = {Math. Ann.},
  fjournal   = {Mathematische Annalen},
  volume     = {359},
  year       = {2014},
  number     = {3-4},
  pages      = {629--636},
  issn       = {0025-5831},
  mrclass    = {20F65},
  mrnumber   = {3231009},
  mrreviewer = {Qin Wang},
  doi        = {10.1007/s00208-014-1007-9},
  url        = {https://doi.org/10.1007/s00208-014-1007-9}
}

@inproceedings{Segal90,
  author     = {Segal, Dan},
  title      = {On the outer automorphism group of a polycyclic group},
  booktitle  = {Proceedings of the {S}econd {I}nternational {G}roup {T}heory
                {C}onference ({B}ressanone, 1989)},
  journal    = {Rend. Circ. Mat. Palermo (2) Suppl.},
  fjournal   = {Rendiconti del Circolo Matematico di Palermo. Serie II.
                Supplemento},
  number     = {23},
  year       = {1990},
  pages      = {265--278},
  issn       = {1592-9531},
  mrclass    = {20F28 (20F16)},
  mrnumber   = {1068367},
  mrreviewer = {D.\ J. S. Robinson}
}

@article{Toinet2013,
  author     = {Toinet, Emmanuel},
  title      = {Conjugacy {$p$}-separability of right-angled {A}rtin groups
                and applications},
  journal    = {Groups Geom. Dyn.},
  fjournal   = {Groups, Geometry, and Dynamics},
  volume     = {7},
  year       = {2013},
  number     = {3},
  pages      = {751--790},
  issn       = {1661-7207,1661-7215},
  mrclass    = {20F36 (20E26 20F28)},
  mrnumber   = {3095717},
  mrreviewer = {R\'emi\ Bernard\ Coulon},
  doi        = {10.4171/GGD/205},
  url        = {https://doi.org/10.4171/GGD/205}
}

@article{Wade_JohnsonHomsHigherRank,
  author     = {Wade, Richard D.},
  title      = {Johnson homomorphisms and actions of higher-rank lattices on
                right-angled {A}rtin groups},
  journal    = {J. Lond. Math. Soc. (2)},
  fjournal   = {Journal of the London Mathematical Society. Second Series},
  volume     = {88},
  year       = {2013},
  number     = {3},
  pages      = {860--882},
  issn       = {0024-6107,1469-7750},
  mrclass    = {20E36 (20F36 22E40)},
  mrnumber   = {3145135},
  mrreviewer = {Fran\c cois\ Dahmani},
  doi        = {10.1112/jlms/jdt044},
  url        = {https://doi.org/10.1112/jlms/jdt044}
}
\bibliographystyle{alpha}

\end{document}